\documentclass{article}
\usepackage{graphicx} 
\usepackage{amsmath,mathtools}
\usepackage{amsthm}
\usepackage{amssymb}
\usepackage{geometry}
\usepackage{hyperref}
\usepackage{subcaption}
\usepackage{booktabs} 
\usepackage{hyperref}
\usepackage{amsfonts}
\usepackage{enumitem}
\usepackage{tikz}
\usetikzlibrary{shapes}
\usetikzlibrary{fit}
\usepackage{thmtools}
\usepackage{thm-restate}
\usepackage{multirow}
\usepackage{authblk}

\newtheorem{theorem}{Theorem}[section]
\newtheorem{corollary}[theorem]{Corollary}

\newtheorem{lemma}[theorem]{Lemma}
\newtheorem{proposition}[theorem]{Proposition}
\theoremstyle{remark}
\newtheorem*{note}{Note}
\newtheorem{remark}[theorem]{Remark}
\newtheorem{claim}[theorem]{Claim}
\newenvironment{claimproof}[1][Proof of Claim]{\begin{proof}[#1]}{\end{proof}}

\title{A New Lower Bound on the Spectral Radius of Graphs with Prescribed Average Degree}
\author[1]{Sonny Ben-Shimon}
\author[2]{Idan Eisner}
\author[3]{Shlomo Hoory}
\date{\today}
\affil[1]{Google Inc.\footnote{sonnyb@google.com}}
\affil[2]{Department of Computer Science, Tel-Hai University of the Galilee, Kiryat Shmona, Israel\footnote{eisnerida@telhai.ac.il}}
\affil[2]{Department of Computer Science, Tel-Hai University of the Galilee, Kiryat Shmona, Israel\footnote{hooryshl@telhai.ac.il}}

\newcommand{\R}{\mathbb{R}}  
\newcommand{\C}{\mathbb{C}}  
\newcommand{\N}{\mathbb{N}} 
\newcommand{\one}{{\bf 1}} 
\newcommand{\dg}{{\bf d}} 
\newcommand{\M}{\mathcal{M}}  
\newcommand{\G}{\mathcal{G}}  
\newcommand{\largeNu}{\text{\scalebox{1.75}{$\nu\!$}}}  
\newcommand{\rhomin}{\rho_{\textrm{min}}}
\newcommand{\rhomax}{\rho_{\textrm{max}}}
\newcommand{\rhorms}{\rho_{\textrm{rms}}}
\newcommand{\rhoent}{\rho_{\textrm{ent}}}
\newcommand{\dmin}{{\textrm{d}}_{\min}}
\newcommand{\dmax}{{\textrm{d}}_{\max}}
\newcommand{\MnR}{M_n(\R)} 
\newcommand{\MnC}{M_n(\C)} 
\DeclareMathOperator{\sgn}{sgn}

\begin{document}

\maketitle

\begin{abstract}
This work establishes an improved lower bound for the spectral radius of a graph given its average degree. 
The new bound follows from an exact solution of the fractional relaxation of the problem. 
Our findings lead to an affirmative answer to a conjecture by Hong (1993) for graphs with specific average degrees -- as the extremal graphs that meet our bound are proven to have a minimal and maximal degree that differ by at most one. Furthermore, we provide an exact characterization of the conditions that permit such discrete realizations. We prove that for a fixed number of vertices $n$, the number of valid edge configurations grows at least linearly with $n$, achieving an average asymptotic order of $\Theta(n\log n)$.
\end{abstract}

\section{Introduction}

Let $G = (V, E)$ be a simple graph with $n$ vertices and $e$ edges. Spectral graph theory studies the properties of a graph by analyzing the eigenvalues of its associated matrices. The primary matrix of interest is the \textbf{adjacency matrix} $A_G$, an $n \times n$ matrix where $A_{ij}=1$ if vertices $v_i$ and $v_j$ are adjacent, and 0 otherwise. Since $A_G$ is real and symmetric, it has $n$ real eigenvalues, $\lambda_1(G) \ge \lambda_2(G) \ge \dots \ge \lambda_n(G)$, which form the \textbf{spectrum} of the graph. The largest of these eigenvalues, $\rho(G) = \lambda_1(G)$, is known as the \textbf{spectral radius} of $G$.  

The spectral radius is not merely an abstract scalar; it is a governing parameter for a multitude of graph-theoretic invariants and processes. The ``Brualdi-Solheid problem''~\cite{brualdi1985spectral},~\cite{brualdi1986spectral} -- the task of determining the graphs that maximize or minimize the spectral radius within a given class of graphs $\mathcal{G}$ -- has become a cornerstone of spectral graph theory. 
We refer the reader to the monograph~\cite{cvetkovic2009introduction} for a multitude of references and applications.

In this work, we are interested in the problem of minimizing the spectral radius $\rho$ of an undirected simple graph $G$, and identifying the extremal graphs which attain these bounds for $n$-vertex graphs with a fixed number of edges $e$. In the process we analyze the family of graphs when the graph must realize a specific, prescribed degree sequence $\dg=(d_1,\dots,d_n)$ where $\sum_i d_i = 2e$. Bi-regular graphs are of special interest. Namely the first $n_1$ vertices are of degree $d_1$ and the other $n_2=n-n_1$ vertices are of degree $d_2$.

We introduce the following notation for the relevant undirected simple graph families. Note that graphs in these families are not necessarily connected.  
\begin{enumerate}
    \item $\G_{n,e}$ are $n$-vertex graphs with $e$ edges.
    \item $\G_\dg$ are graphs with a given degree sequence $\dg=(d_1,\dots,d_n)$. 
    \item $\G_{d_1,n_1,d_2,n_2}$ are bi-regular graphs with $n_1$ degree $d_1$ vertices and $n_2$ degree $d_2$ vertices, where $n=n_1+n_2$.
\end{enumerate}

For a graph family $\G$, We let $\rhomin(\G)$ and $\rhomax(\G)$ denote the minimal and maximal spectral radius $\rho(G)$ for $G \in \G$, respectively. As this work is primarily interested in the minimal spectral radius, we denote $\rho(\G) = \rhomin(\G)$.

\subsection{Upper bounds}
The maximization problem has been extensively studied, often yielding results that point toward nested split graphs or graphs with dominant high-degree vertices (such as stars or cliques). In~\cite{stanley1987bound} Stanley provided a sharp upper bound depending solely on the number of edges $e$.
$$\rhomax(\G_{n,e}) \le \frac{-1 + \sqrt{1 + 8e}}{2}$$
Equality is attained if and only if $e = \binom{k}{2}$ and $G$ is the disjoint union of a complete graph $K_k$ and isolated vertices.

Hong~\cite{Hong1988} improved this by establishing the sharp upper bound $$\rhomax(\G_{n,e}) \le \frac{-1 + \sqrt{1 + 8e}}{2}.$$
The maximization problem has been extensively studied, often yielding results that point toward nested split graphs or graphs with dominant high-degree vertices (such as stars or cliques). In~\cite{stanley1987bound} Stanley provided a sharp upper bound depending solely on the number of edges $e$, improving upon an earlier result~\cite{brualdi1985spectral}.
Equality is attained if and only if $e = \binom{k}{2}$ and $G$ is the disjoint union of a complete graph $K_k$ and isolated vertices.

This was improved by Hong~\cite{Hong1988}, who established the sharp upper bound 
$$\rhomax(\G_{n,e}) \le \sqrt{2e - n + 1}$$
for any connected graph. 
Subsequently, increasingly precise bounds incorporating the degree sequence have been developed by Hong~\cite{hong2001sharp}, Shu and Wu~\cite{shu2004sharp}, and Liu and Weng~\cite{liu2013spectral}. Furthermore, Nikiforov~\cite{Nikiforov2002, Nikiforov2011} significantly advanced this field by systematically bridging spectral extremal theory with classical graph-theoretic problems.

\subsection{Lower bounds}\label{sect:lower_bounds}
For the lower bound there has been less work as the minimization problem presents a significantly more subtle and intricate challenge. Before delving into the main results of this work, we summarize the known lower bounds on the spectral radius for the aforementioned graph families.

\begin{proposition}[Theorem 3.2.1 in~\cite{cvetkovic2009introduction}]\label{prop:average_deg_lb}
$$\rhomin(\G_{n,e}) \ge \overline{d}= 2e/n.$$
\end{proposition}

\begin{proposition}[Hoffman, Theorem 8.1.25 in~\cite{cvetkovic2009introduction}]\label{prop:deg_rms_lb}
    $$\rhomin(\G_\dg) \ge \sqrt{\sum_{i=1}^n d_i^2/n}.$$
\end{proposition}

The following is an entropy-based bound which follows from the fact that the Maximal Entropy Random Walk achieves at least the entropy of any (memoryless) random walk on the graph, in particular, that of the Simple Random Walk. It can also be derived directly using techniques similar to those of~\cite{alon2002moore,eisner2024entropy}.
\begin{proposition}[Burda et. al, \cite{Burda2009} eq. (14)]\label{proposition:entropy_lb}   
    $$\rhomin(\G_\dg) \ge \prod_{i=1}^n d_i^{\frac{d_i}{2e}}.$$
\end{proposition}
 
It should be noted that the root mean square bound \ref{prop:deg_rms_lb} as well as the entropy bound \ref{proposition:entropy_lb} are better than the simple average degree bound \ref{prop:average_deg_lb} with equality if and only if the graph is regular. Furthermore, all bounds apply to the spectral radius of general symmetric non-negative matrices, by virtually the same proofs.

\subsection{Our contribution}

The problem of minimizing the spectral radius of a graph in $\G_{n,e}$ is closely related to a conjecture by Hong (1993), problem 3 in ~\cite{hong1993bounds}, 
suggesting that if a simple connected graph with a prescribed number of edges and vertices has a minimal spectral radius, then the difference between its maximal and minimal vertex degrees is at most one
\footnote{The original conjecture, which requires the graph to be irregular even when the average degree is integral, was demonstrated to be false by Reti~\cite{reti2018graph}.}.
Recently, Cioaba et al.~\cite{cioaba2024minimum} proved that for some cases of dense graphs the conjecture is true,  where $(e, n)$ satisfies $e \geq (n-1)(n-2)/2 - 2$ or $e = n^2/4 - 1$ or $e = n^2/3 - 1$.

Our approach is to relax the problems to weighted graphs by dropping the integrality constraint on the adjacency matrix. 
Rather than directly investigating $\rho(\G) = \min\{\rho(A_G) \mid G\in\G\}$ 
when $\G$ is either $\G_{n,e}$ or $\G_{d_1,n_1,d_2,n_2}$, 
we consider $\rho(\M) = \min\{\rho(M) \mid M\in\M\}$, for two specific sets of non-negative symmetric $n \times n$ matrices $\M$:
\begin{enumerate}
    \item matrices with a total sum of $2e$ and integral row sums, denoted by $\M_{n,e}$,
    \item matrices with $n_1$ sum $d_1$ rows and $n_2$ sum $d_2$ rows, denoted by $\M_{d_1,n_1,d_2,n_2}$, where $d_1, d_2$ are any distinct non-negative real numbers.
\end{enumerate}

\smallskip\noindent
Note that the two sets $\M_{n,e}$ and $\M_{d_1,n_1,d_2,n_2}$ are compact, so the minimum of $\rho$ is attained.
Our first result is a formula for the minimal $\rho$ on each of the two sets, as well as an exact characterization of the weighted graphs achieving that minimum. 
The result for $\M_{n,e}$ can be regarded as an affirmative answer to the relaxation of Hong's question to weighted graphs.

\begin{restatable}{theorem}{BiregularMatrixTheorem}\label{theorem:Mn1d1n2d2}
    Given integers $n_1 \geq 0$, $n_2>0$ and non-negative distinct real numbers $d_1,d_2$ such that $n_1 d_1 \le n_2 d_2$, 
    \begin{equation}\label{eq:rhoopt_intro}
       \rho(\M_{d_1,n_1,d_2,n_2}) = 
            \frac{1}{2}\left[d_2 - d_1 \frac{n_1}{n_2} + \sqrt{4d_1^2 \frac{n_1}{n_2} + (d_2 - d_1 \frac{n_1}{n_2})^2}\,\right],
    \end{equation} 
    where a weighted graph in $\M_{d_1,n_1,d_2,n_2}$ is $\rho$-minimizing iff its induced graph on the degree $d_1$ vertices is empty and its induced graph on the degree $d_2$ vertices is regular.
\end{restatable}

\begin{restatable}{theorem}{MneTheorem}\label{theorem:Mne}
    Given integers $n>0$ and $e \geq 0$, if a matrix $M \in \M_{n,e}$ satisfies $\rho(M) = \rho(\M_{n,e})$, 
    then $M \in \M_{d_1,n_1,d_2,n_2}$, 
    where $d_1 = \lfloor 2e/n \rfloor$, $d_2 = d_1+1$, $n_2 = 2e \bmod n$ and $n_1 = n-n_2$.
\end{restatable}
 
As the two results are invariant under scaling of the number of edges and vertices, $\rho(\M_{n,e})$ depends only on the average degree $\overline{d}=\frac{2e}{n}$, while $\rho(\M_{d_1,n_1,d_2,n_2})$ depends only on $d_1,d_2$ and $\nu=\frac{n_2}{n_1+n_2}$.
In particular, $\rho(\M_{n,e}) = \overline{d}$ when the average degree is an integer.
Otherwise the minimum is obtained on a bi-regular weighted graph where the minimal and maximal degrees differ by one. 

The formulas for $\rho(\M_{n,e})$ and $\rho(\M_{d_1,n_1,d_2,n_2})$ establish new lower bounds on $\rho(\G_{n,e})$ and $\rho(\G_{d_1,n_1,d_2,n_2})$ respectively. 
These new bounds improve upon Hoffman's root mean square bound and upon the entropy based bound in those cases.

Figure~\ref{fig:rho_opt} depicts the new lower bound $\rho(\M_{n,e})$ as a function of the average degree $\overline{d}$. 
The graph on the right depicts $\rho(\M_{n,e})-\overline{d}$ to show the details more clearly,
where $\overline{d}$ is the trivial lower bound on $\rho(\M_{n,e})$. It should be noted that $\rho(\M_{n,e})=1$ for $0 < \overline{d} \leq 1$, 
and that $\rho(\M_{n,e})$ is discontinuous at zero. The gap between the spectral radius and $\overline{d}$ is known as the Collatz–Sinogowitz irregularity index~\cite{von1957spektren}. 
It works well in our figures and it emerges naturally in Theorem~\ref{theorem:rhoopt_residue}, where we establish the precise scaling law of $\rho(\M_{n,e})-\overline{d}$ for large values of $\overline{d}$.
\begin{figure}[h!]
    \begin{subfigure}{0.5\textwidth}
        \includegraphics[width=0.9\textwidth]{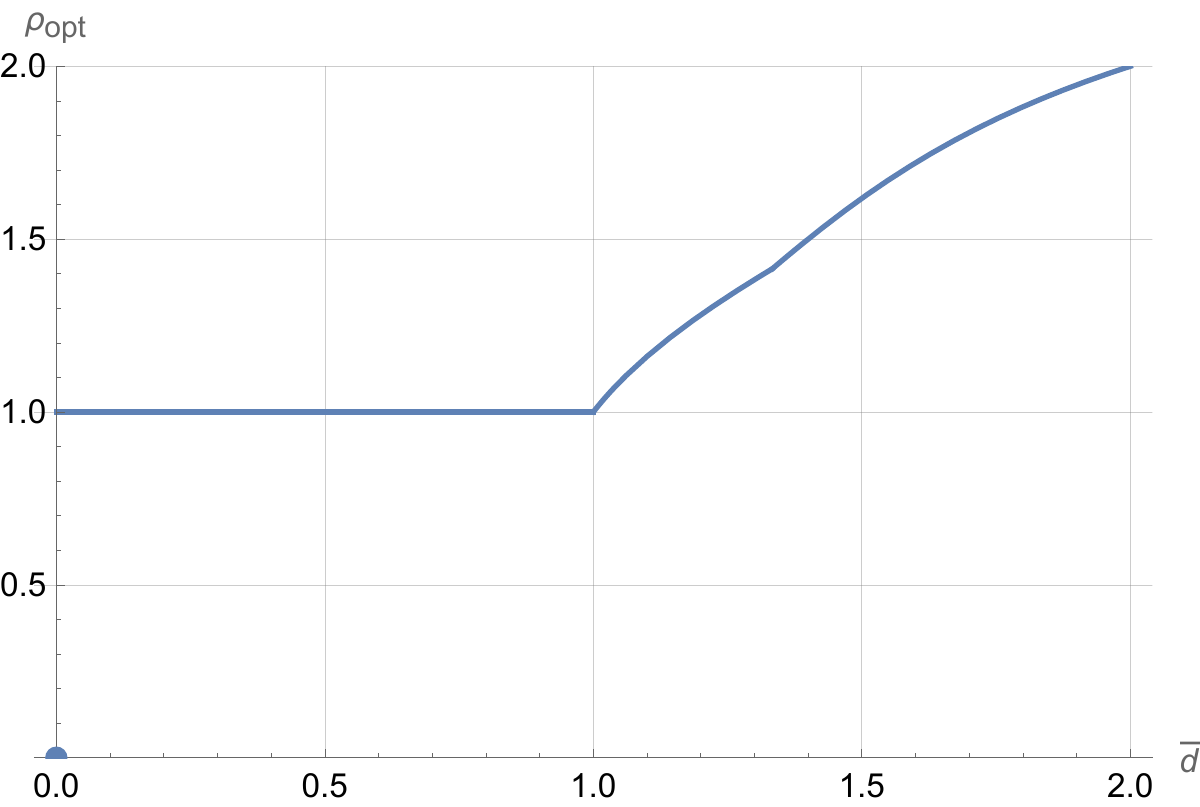}
        \caption{$\rho(\M_{n,e})$ as a function of $\overline{d}$}
    \end{subfigure}\hfill
    \begin{subfigure}{0.5\textwidth}
        \includegraphics[width=\textwidth]{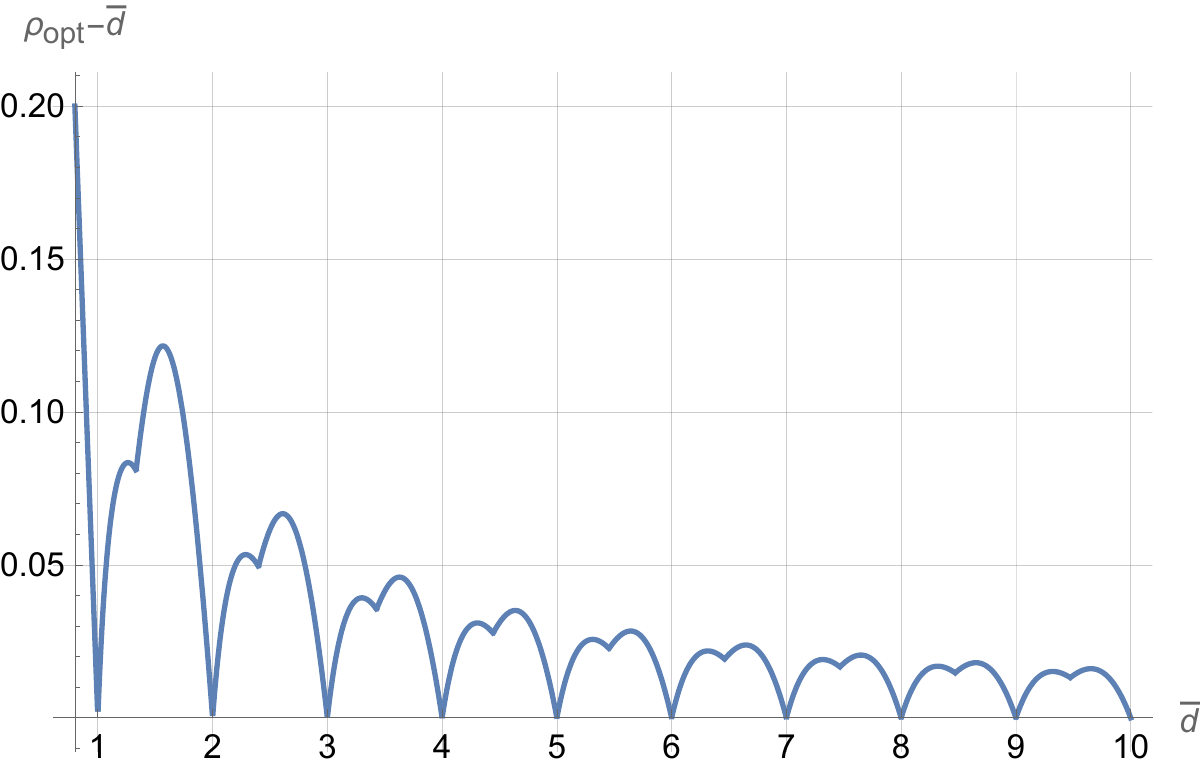}
        \caption{$\rho(\M_{n,e}) - \overline{d}$ as a function of $\overline{d}$}
    \end{subfigure}
    \caption{$\rho(\M_{n,e})$ as a function of $\overline{d} = 2e/n$.}
    \label{fig:rho_opt}
\end{figure}

The last results ask whether the minimum $\rho$ on a weighted graph family is in fact achievable by a simple graph in that family.
For the family $\M_{n,e}$, an affirmative answer implies that Hong's conjecture is true for the specific $n,e$ pair.
By Theorem~\ref{theorem:Mne}, the $\rho$-minimizing weighted graphs in $\M_{n,e}$ are either regular or bi-regular with degree difference of one, implying that the integers $d_1, n_1, d_2, n_2$ are uniquely determined by $n,e$. The following theorem gives an exact characterization of such $n,e$ pairs. 

\begin{restatable}{theorem}{WhenBiregularSimpleGraphTheorem}\label{theorem:when_biregular_simple_graph}
    Let $n_1, n_2$ be positive integers and $d_1, d_2$ be non-negative integers such that $d_1 \neq d_2$ and $n_2 d_2 \ge n_1 d_1$,
    then $\rho(\G_{d_1,n_1,d_2,n_2})=\rho(\M_{d_1,n_1,d_2,n_2})$, 
    if and only if the following conditions hold: (i) $n_2 d_2 + n_1 d_1$ is even, (ii) $n_2$ divides $n_1 d_1$, (iii) $n_2 \ge d_1$, and (iv) $n_2(n_2-1) \ge n_2 d_2 - n_1 d_1$.
\end{restatable}

It should be noted that a necessary and sufficient condition for the existence of a $\rho$-minimizing graph that is also connected can be characterized.
See Theorem~\ref{theorem:when_biregular_simple_connected_graph} for more details. 

Although the above theorem gives a necessary and sufficient condition for the existence of such graphs, it is rather opaque.
The following two theorems are an attempt at a fuller characterization of such $n,e$ pairs. See Section~\ref{sect:disrete_graphs} for more details.

\begin{restatable}{theorem}{NuSimpleGraphTheorem}\label{theorem:nu_simple_graph}
    Let $d_1, d_2$ be two positive integers such that $d_1 \neq d_2$, and let $\nu$ be a rational number in the open interval $(0,1)$. 
    Then, a necessary and sufficient condition for the existence of a positive integer $n$ such that $\rho(\G_{d_1,n_1,d_2,n_2})=\rho(\M_{d_1,n_1,d_2,n_2})$ with $n_1=(1-\nu)n$ and $n_2 = \nu n$, is that $\nu$ is in the set
    \begin{equation}\label{eq:nu_set}
        \largeNu_{d_1,d_2} = 
            \left\{ \frac{i}{d_2+i} : i = 1,\ldots,d_1-1\right\} 
                \cup  \left\{ \frac{d_1}{d_1+i} : i = 1,\ldots,d_2\right\}.
    \end{equation}
    If $\nu \in \largeNu_{d_1,d_2}$, then $\rho(\G_{d_1,n_1,d_2,n_2})=\rho(\M_{d_1,n_1,d_2,n_2})$ holds for infinitely many values of $n$.
\end{restatable}

\begin{note}
    Theorem~\ref{theorem:Mne} implies that for a graph $G$ with an average degree $\overline{d} = \frac{2e}{n}\ge 1$ such that $\rho(G) = \rho(\M_{n,e})$, the difference between the minimal and maximal degrees is at most one.
    Applying Theorem~\ref{theorem:nu_simple_graph} with $d_1=\lfloor \overline{d} \rfloor$ and $d_2 = d_1+1$ implies that such a graph exists if and only if $\overline{d}$ is in the following list:
    \begin{equation}\label{eq:average_degree_sequence}
      1,\, 1\tfrac{1}{3}, \, 1\tfrac{1}{2}, \, 2, \, 2\tfrac{1}{4}, \, 2\tfrac{2}{5}, \, 2\tfrac{1}{2}, \, 2\tfrac{2}{3}, \, 3, \, 3\tfrac{1}{5}, \,  3\tfrac{1}{3}, \, 3\tfrac{3}{7}, \, 3\tfrac{1}{2}, \, 3\tfrac{3}{5}, \, 3\tfrac{3}{4},\,\ldots
    \end{equation}
\end{note}

Finally, we study the number of $(n,e)$ pairs for which an optimal continuous relaxation can be realized by a discrete simple graph, thereby providing an affirmative answer to Hong's conjecture for those specific pairs.
In~\cite{cioaba2024minimum} Cioab\u{a} et al. successfully proved Hong's conjecture for a linear number of edge values, primarily in dense regimes where $e \ge \binom{n-1}{2}-2$ alongside specific sparse cases like $e = \frac{n^2}{4}-1$. This leaves Hong's original question open for the vast majority of $(n,e)$ pairs.  

Let $E(n)$ denote the number of non-trivial edge values $e$ (where $n$ does not divide $2e$) such that $\rho(\G_{n,e}) = \rho(\M_{n,e})$.
In Section~\ref{section:num_edges_simple_graph}, we establish an exact number-theoretic formula for $E(n)$ governed by Pillai's arithmetical function. Rather than presenting the full algebraic formulation here, we state its primary consequences. We demonstrate that $E(n)$ always grows at least linearly, but experiences super-linear growth for highly composite numbers.

\begin{restatable}{theorem}{NumEdgesSimpleGraphTheorem}\label{theorem:num_edges_simple_graph}
    For any integer $n \ge 3$, let $E(n)$ be the number of distinct values of $e$ such that $n$ does not divide $2e$ and $\rho(\G_{n,e}) = \rho(\M_{n,e})$. Then the following hold:
    \begin{enumerate}
        \item $E(n) \ge \lfloor\frac{3n-5}{2}\rfloor$.
        \item The lower bound is tight if and only if $n=p$ or $n=2p$ for some prime number $p$.
        \item For an infinite family of highly composite, square-free integers, $E(n) = \Omega(n \log n)$.
    \end{enumerate}
\end{restatable}

This super-linear scaling demonstrates that exact continuous relaxations can be realized as discrete simple graphs at a surprising rate across the density spectrum. 
While Theorem~\ref{theorem:num_edges_simple_graph} isolates the maximal peaks of $E(n)$ for highly composite numbers, the asymptotic average of Pillai's function guarantees that $E(n)$ grows as $\Theta(n \log n)$ on average, see Remark~\ref{remark:average_num_edges_simple_graph}. 
This still leaves a gap compared to the $O(n^2)$ possible values for $e$, leaving Hong's general conjecture open for the majority of pairs.

\section{Weighted Graphs Preliminaries}

In this section we set up the stage for the relaxation of the minimal spectral radius problem to weighted graphs.
Namely, symmetric non-negative matrices with some row sum constraints. We use matrix and graph terminology interchangeably throughout this work, where vertices are row numbers, matrix entries are edge weights, and vertex degrees are row sums. We say that an edge exists if the corresponding entry is non-zero, that the graph is connected if the matrix is irreducible, and that the matrix blocks are the connected components of the graph. 
We let $\MnR, \MnC$ denote the set of $n \times n$ matrices over the respective field, and $[n]$ denote the set of integers from $1$ to $n$. 
We are interested in the following families of weighted graphs on $n$ vertices:
\begin{itemize}
    \item $\M_{n,e}$
    where $e$ is a non-negative integer and the degrees are required to be integral. The total sum of matrix entries is $2e$, or equivalently, the average degree is $2e/n$.
    \item $\M_{d_1,n_1,d_2,n_2}$, bi-regular graphs, where there are $n_1$ degree $d_1$ vertices and $n_2$ degree $d_2$ vertices, and $n_1,n_2$ are non-negative integers with $n=n_1+n_2$ and $d_1,d_2$ are distinct non-negative real numbers. Given a graph $G$ in this family, we usually refer to the set of degree $d_i$ vertices as $V_i$ for $i=1,2$.
    \item $\M_\dg$, graphs with degree sequence $\dg=(d_1,\ldots,d_n)$, where the real number $d_i \geq 0$ is the degree of vertex $i$.
\end{itemize}

It should be noted that for $\M$ being any of the above matrix sets, the minimal spectral radius $\rho(\M)$ is attained, as these are compact sets. We start with some fundamental facts from matrix theory and build upon them towards the proofs of our main results.

\begin{proposition}[Perron–Frobenius, Theorem 8.4.4 from~\cite{horn2012matrix}]\label{proposition:perron}
    Given a non-negative irreducible matrix $M \in \MnR$ for $n \ge 2$ the spectral radius $\rho(M)$ is a positive and algebraically simple eigenvalue.
    Furthermore, up to scaling there are unique vectors $x,y$ so that $M x = \rho(M) x$ and $y^T M = \rho(M) y^T$;
    these vectors are strictly positive.
\end{proposition}

\begin{proposition}  
    The function $\rho$ is convex on the set of symmetric matrices in $\MnR$.
\end{proposition}

\begin{proof}
    For any constant $c \geq 0$, we have $\rho(c M) = c \rho(M)$.
    Furthermore, as $M$ is a normal matrix,
    $\rho(M) = \|M\|_2 = \underset{\|x\|_2=1}{\max} \|Mx\|_2$, implying that $\rho$ is sub-additive, 
    \begin{equation*}\begin{split}
        \rho(M_1+M_2) &= \underset{\|x\|_2=1}{\max} \|(M_1+M_2)x\|_2 \\
        &\le \underset{\|x\|_2=1}{\max} \|M_1x\|_2 + \underset{\|x\|_2=1}{\max} \|M_2x\|_2 = \rho(M_1) + \rho(M_2).
    \end{split}\end{equation*}
\end{proof}

\begin{lemma}\label{lemma:two_matrices_same_Perron}
        Let $\M \subset \MnR$ be a set of non-negative symmetric matrices.
        Let $M_0, M_1 \in \M$ be two irreducible matrices such that $\rho(M_0) = \rho(M_1)= \rho(\M)$, and assume that
        $M^{(\alpha)} = (1-\alpha) M_0 + \alpha M_1 \in \M$ for some $0 < \alpha < 1$. 
        Then $M_0, M_1$ and $M^{(\alpha)}$ have the same Perron eigenvector.
\end{lemma}

\begin{proof}
    Let $\rhomin = \rho(\M)$. 
    Then by the convexity of $\rho$ on $\M$,
    \begin{equation*}
        \rhomin \le \rho(M^{(\alpha)}) \le (1-\alpha) \rho(M_0) + \alpha \rho(M_1) = \rhomin.
    \end{equation*}
    Let $x^{(\alpha)}$ be the Perron eigenvector of $M^{(\alpha)}$, where $\|x^{(\alpha)}\|=1$.
    Then by the convexity of the $l_2$ norm:
    \begin{equation*}
        \rhomin = \|M^{(\alpha)} x^{(\alpha)}\| \le (1-\alpha) \|M_0 x^{(\alpha)}\| + \alpha \|M_1 x^{(\alpha)}\|.
    \end{equation*}

    Since for a symmetric matrix the operator norm is equal to the spectral radius, it follows that $\|M_i x^{(\alpha)}\| \le \rho(M_i)=\rhomin$, for $i=0,1$, implying that
    \[
        \rhomin = \|M^{(\alpha)} x^{(\alpha)}\| = \|M_0 x^{(\alpha)}\| = \|M_1 x^{(\alpha)}\|.
    \]
    Therefore, by the strict convexity of Euclidean norm the three vectors are equal
    \[
        M^{(\alpha)} x^{(\alpha)} = M_0 x^{(\alpha)} = M_1 x^{(\alpha)}.
    \]
    It follows that $x^{(\alpha)}$ is the Perron eigenvector of $M_0$ and $M_1$ as well as $M^{(\alpha)}$.
\end{proof}

\begin{corollary}\label{corollary:x_fixed_on_same_sum_rows}
    Let $\M$ be a set of symmetric non-negative matrices closed under degree-preserving vertex permutations (such as $\M_{n,e}$ or $\M_{d_1,n_1,d_2,n_2}$). 
    Let $M \in \M$ be an irreducible matrix with $\rho(M) = \rho(\M)$ and let $x$ be the Perron eigenvector of $M$. 
    Then any two vertices $i,j$ with the same degree satisfy $x_i = x_j$.
\end{corollary}

\begin{proof}
    Apply Lemma~\ref{lemma:two_matrices_same_Perron} with 
    $M_0=M$, $M_1=P^{-1}MP$ and $\alpha=1/2$, 
    where $P$ is the permutation matrix swapping rows $i$ and $j$. 
\end{proof}

\begin{proposition}[Theorem 6.3.12 from~\cite{horn2012matrix}]\label{proposition:lambda_perturbation}
    Given the matrices $A, E \in \MnC$ with a simple eigenvalue $\lambda(A)$, then the function $\lambda(A + t E )$ is well defined and differentiable with respect to $t$ at $t=0$ with:
    \begin{equation*}
        \frac{d \lambda(A + t E)}{dt}\bigg\rvert_{t=0} = \frac{x^*Ey}{x^*y},
    \end{equation*}
    where $x$ and $y$ are the left and right eigenvectors corresponding to $\lambda$, $x^*A = \lambda(A)x^*$ and $Ay=\lambda(A)y$.  
\end{proposition}

\begin{corollary}\label{corollary:rho_perturbation}
        If $A \in \MnR$ is a symmetric non-negative irreducible matrix, $E \in \MnR$ is symmetric and there exists $\epsilon>0$ such that $A + t E$ is non-negative and irreducible for all $t \in [0,\epsilon)$,
        then the function $\rho(A + t E )$ is differentiable with respect to $t$ at $t=0^+$ with:
        \begin{equation*}
            \frac{d \rho(A + t E)}{dt}\bigg\rvert_{t=0^+} = \frac{x^TEx}{\|x\|^2},
        \end{equation*}
        where $x$ is the Perron eigenvector of $A$, satisfying $A x = \rho(A) x$.      
\end{corollary}

\begin{proof}
    Let $A, E$ and $\epsilon$ satisfy the stated assumptions.
    By Proposition~\ref{proposition:perron}, $\rho(A) = \lambda(A)$ is a simple eigenvalue of $A$.
    Proposition~\ref{proposition:lambda_perturbation} then implies that the function $\lambda(A+tE)$ is well defined on the interval $(-\delta,\delta)$ for some $\delta > 0$.
    Since $A+tE$ is nonnegative and irreducible for every $t \in \left[0,\min(\epsilon,\delta)\right)$, we have $\rho(A+tE)=\lambda(A+tE)$.
    Hence, Proposition~\ref{proposition:lambda_perturbation} yields the desired result.
\end{proof}

\begin{corollary}\label{corollary:derivative_Eij}
    Let  $\M$ be a  set of symmetric non-negative matrices closed under degree-preserving perturbations (such as $\M_{n,e}$ or $\M_{d_1,n_1,d_2,n_2}$). Let $M \in \M$ be an irreducible matrix with Perron eigenvector $x$, such that $M_{ii}, M_{jj} > 0$ for some $i \neq j$.
    Let $E=E^{(ij)} \in \MnR$ be a matrix that is all zero except that $E_{ii}=E_{jj}=-1$ and $E_{ij}=E_{ji}=1$.
    Then there exists $\epsilon > 0$ satisfying $M + t E \in \M$ for all $t \in [0,\epsilon)$ and:
    \begin{equation*}
        \frac{d}{dt} \rho(M + t E)\bigg\rvert_{t=0^+}  = -\frac{(x_i-x_j)^2}{\|x\|^2}.
    \end{equation*}
\end{corollary}

\begin{proof}
    Given matrices $M, E=E^{(ij)}$ satisfying the requirements, 
    the condition $M + t E \in \M$ holds for $t \in [0,\epsilon]$, 
    where $\epsilon = \min(M_{ii}, M_{jj})$.
    The result follows by Corollary~\ref{corollary:rho_perturbation}.
\end{proof}

\begin{lemma}\label{lemma:generalized_perturbation}
    Let  $\M$ be a  set of symmetric non-negative matrices closed under degree-preserving perturbations (such as $\M_{n,e}$ or $\M_{d_1,n_1,d_2,n_2}$). Let $M \in \M$ be an irreducible matrix with $\rho(M) = \rho(\M)$, and let $x$ be its Perron eigenvector. Suppose there exist disjoint vertex sets $U$ and $W$ such that $x$ is constant on $U$ with value $\chi_U$, and constant on $W$ with value $\chi_W$. If $\chi_U \neq \chi_W$, then at least one of the induced subgraphs $M_{|U}$ and $M_{|W}$ must be empty.
\end{lemma}

\begin{proof}
    Suppose for the sake of contradiction that both induced graphs $M_{|U}$ and $M_{|W}$ are non-empty. If there exist distinct indices $i_1, i_2 \in U$ such that $M_{i_1 i_2} > 0$, and distinct indices $j_1, j_2 \in W$ such that $M_{j_1 j_2} > 0$, we define the perturbation matrix 
    $E = E^{(i_1 j_1)} - E^{(i_1 i_2)} - E^{(j_1 j_2)}$.
    As $M + tE \in \M$ for a sufficiently small $t > 0$, we may apply Corollary~\ref{corollary:rho_perturbation} to obtain that
    \begin{equation*}
        \frac{d}{dt}\rho(M + tE)\bigg\rvert_{t=0^+} 
        = \frac{1}{\|x\|^2} \left[ (x_{i_1} - x_{i_2})^2 + (x_{j_1} - x_{j_2})^2 - (x_{i_1} - x_{j_1})^2 \right] 
        = -\frac{(\chi_U - \chi_W)^2}{\|x\|^2} 
        < 0,
    \end{equation*}
    where the last equality holds as $x_{i_1} = x_{i_2} = \chi_U$ and $x_{j_1} = x_{j_2} = \chi_W$.
    This strictly negative derivative contradicts the assumption that $M$ minimizes the spectral radius on $\M$. 
    
    If $M_{|U}$ has only self-loop edges, we choose $M_{i_1 i_1} > 0$ for some $i_1 \in U$ and delete the term $E^{(i_1 i_2)}$ from the definition of $E$. Similarly, if $M_{|W}$ has only self-loop edges, we choose $M_{j_1 j_1} > 0$ for some $j_1 \in W$ and delete the term $E^{(j_1 j_2)}$ from the definition of $E$. In all cases, a strictly negative derivative is obtained, contradicting the minimality of $\rho(M)$. Therefore, at least one of the induced graphs $M_{|U}$ or $M_{|W}$ must be empty.
\end{proof}

\section{The Bi-Regular Lower Bound}

In this section, we prove Theorem~\ref{theorem:Mn1d1n2d2}, stating that for the integers $n_1 \geq 0$, $n_2>0$ and non-negative real numbers $d_1,d_2$ such that $n_1 d_1 \le n_2 d_2$, equation \eqref{eq:rhoopt_intro} holds. 
Namely, $\rho(\M_{d_1,n_1,d_2,n_2}) = \rho_0(d_1,n_1,d_2,n_2)$, where $\rho_0$ is given by
\begin{equation}\label{eq:rho0_n1d1n2d2}
    \rho_0(d_1,n_1,d_2,n_2) = 
        \frac{1}{2}\left[d_2 - d_1 \frac{n_1}{n_2} + \sqrt{4d_1^2 \frac{n_1}{n_2} + (d_2 - d_1 \frac{n_1}{n_2})^2}\,\right].
\end{equation}

\noindent
Some notes about this result:
\begin{enumerate}
    \item 
        If $d_1 n_1 > d_2 n_2$, the roles of $d_1,n_1$ and $d_2,n_2$ should be exchanged, so Theorem~\ref{theorem:Mn1d1n2d2} and the rest of the results of this section can be applied.
    \item
        If $n_1=0$ this is the $d_2$-regular case, and as expected $\rho_0 = d_2$.
    \item 
        If $d_1 n_1 = d_2 n_2$, then $\rho_0=\sqrt{d_1d_2}$, as expected for a bi-partite bi-regular graph.
\end{enumerate}
 
As the function $\rho_0$ is invariant under scaling of $n_1$ and $n_2$, it is natural to let $\nu=\frac{n_2}{n_1+n_2}$.
Also, as the condition $n_1 d_1 \le n_2 d_2$ translates to $\frac{d_1}{d_1+d_2} \le \nu \le 1$, 
we derive the case $0 \le \nu < \frac{d_1}{d_1+d_2}$ by exchanging $d_1$ and $d_2$ and $\nu$ with $1-\nu$. 
This results with the following expression for $\rho(d_1,d_2,\nu)$:
\begin{equation}\label{eq:rho0_d1d2nu}
    \rho_0(d_1,d_2,\nu) = \begin{cases}\scriptstyle
        \frac{1}{2}\left[d_2 - d_1 \frac{1-\nu}{\nu} + \sqrt{4d_1^2 \frac{1-\nu}{\nu} + (d_2 - d_1 \frac{1-\nu}{\nu})^2}\,\right] & \textrm{ if }\frac{d_1}{d_1+d_2} \le \nu \le 1 \\
        \addlinespace[0.3cm]
        \scriptstyle
        \frac{1}{2}\left[d_1 - d_2 \frac{\nu}{1-\nu} + \sqrt{4d_2^2 \frac{\nu}{1-\nu} + (d_1 - d_2 \frac{\nu}{1-\nu})^2}\,\right] & \textrm{ if }0 \le \nu < \frac{d_1}{d_1+d_2}
    \end{cases}
\end{equation}

We start with two lemmas required for the proof of Theorem~\ref{theorem:Mn1d1n2d2}.

\begin{lemma}\label{lemma:bounded_by_rho0}
        Given the integers $n_1 \geq 0$, $n_2>0$ and distinct non-negative real numbers $d_1,d_2$ such that $n_1 d_1 \le n_2 d_2$, we have
        \begin{equation*}
            \rho(\M_{d_1,n_1,d_2,n_2}) \le \rho_0(d_1,n_1,d_2,n_2).
        \end{equation*}
\end{lemma}

\begin{proof}
    To prove the lemma, we exhibit a matrix $M$ in $\M_{d_1,n_1,d_2,n_2}$ and show that $\rho(M) = \rho_0(d_1,n_1,d_2,n_2)$.
    Let $c_{k \times l}$ denote the $k \times l$ matrix with all entries equal to $c$. 
    Given $n_1,n_2,d_1,d_2$ satisfying the requirements, let $V_1=[n_1]$ and $V_2=[n_1+n_2]\setminus[n_1]$.
    Define the matrix $M \in \M_{d_1,n_1,d_2,n_2}$ by
    \begin{equation*}
        M = \begin{pmatrix}
          0_{n_1 \times n_1}                  & (\frac{d_1}{n_2})_{n_1 \times n_2} \\
          (\frac{d_1}{n_2})_{n_2 \times n_1}  & \frac{1}{n_2}\left(d_2 - n_1 \frac{d_1}{n_2}\right)_{n_2 \times n_2}
        \end{pmatrix}        
    \end{equation*}
    If $d_1=0$, then $M = 0_{n_1 \times n_1} \oplus (\frac{d_2}{n_2})_{n_2 \times n_2}$, so $\rho(M) = d_2 = \rho_0(d_1,n_1,d_2,n_2)$, as claimed.
    Otherwise, $M$ is irreducible. Let $x$ be its Perron eigenvector, $Mx = \rho(M) x$.
    Then by Corollary~\ref{corollary:x_fixed_on_same_sum_rows}, $x$ is constant on $V_1$ and on $V_2$. 
    Let $\chi_l$ be the value of $x$ on $V_l$ for $l=1,2$.
    Then, letting $i$ be an arbitrary index in $V_1$ and $j$ an index in $V_2$, we have
    \begin{align*}
        \rho(M) \chi_1 & = \rho(M) x_i = (M x)_i = \chi_1 n_1 \cdot 0 + \chi_2 n_2 \cdot \frac{d_1}{n_2} = \chi_2 d_1 \\
        \rho(M) \chi_2 & = \rho(M) x_j =  (M x)_j = \chi_1 n_1 \cdot \frac{d_1}{n_2} + \chi_2 n_2 \cdot \frac{1}{n_2} \left(d_2 - n_1 \frac{d_1}{n_2}\right) \\
                       & = \chi_1 d_1 \frac{n_1}{n_2} + \chi_2 \left(d_2 - d_1 \frac{n_1}{n_2}\right).
    \end{align*}
    Therefore
    \begin{equation*}
        M' \begin{pmatrix}\chi_1\\\chi_2\end{pmatrix} = \rho(M)\begin{pmatrix}\chi_1\\\chi_2\end{pmatrix}, \quad\text{ where } M' = \begin{pmatrix} 0 & d_1 \\ d_1 \frac{n_1}{n_2} & d_2 - d_1 \frac{n_1}{n_2}\end{pmatrix},
    \end{equation*}
    implying that $\rho(M) = \rho(M')$. 
    Solving the quadratic equation for the eigenvalues of $M'$ yields that its spectral radius is $\rho_0(d_1,n_1,d_2,n_2)$.
\end{proof}

\begin{lemma}\label{lemma:rho0_monotone}
    The function $\rho_0(d_1,d_2,\nu)$, for $0 \le \nu \le 1$, is 
    strictly increasing in $\nu$ if $0 < d_1 < d_2$, and strictly decreasing in $\nu$ if $0 < d_2 < d_1$.
\end{lemma}

\begin{proof}
    Our first step is to restrict the domain of $\nu$ to $\left[\frac{d_1}{d_1+d_2}, 1\right]$.
    As equation \eqref{eq:rho0_d1d2nu} satisfies the identity $\rho_0(d_1,d_2,\nu) = \rho_0(d_2,d_1,1-\nu)$ for all $0 \le \nu \le 1$, 
    if $0 \le \nu < \frac{d_1}{d_1+d_2}$, we use the lemma with $d_2$, $d_1$, and $1-\nu$ instead. 
    Since this transformation changes the order of $d_1, d_2$ as well as switches between increasing and decreasing, this will prove the lemma also in this case.
    
    It is more convenient to prove the claim for $x=\frac{1-\nu}{\nu}$ where $0 \le x \le \frac{d_2}{d_1}$, noting that $x$ is monotone decreasing in $\nu$. Then we may write:
    \begin{equation*}
        \rho_0(d_1,d_2,x) = 
            \frac{1}{2}\left[d_2 - d_1 x + \sqrt{4d_1^2 x + (d_2 - d_1 x)^2}\,\right].
    \end{equation*}
    Therefore,    
    \begin{equation*}\begin{split}
        \frac{\partial\rho_0(d_1,d_2,x)}{\partial x} 
        &= \frac{1}{2}\left[-d_1 + \frac{4d_1^2 - 2d_1(d_2 - d_1 x)}{2\sqrt{4d_1^2 x + (d_2 - d_1 x)^2}}\,\right] \\
        &= \frac{d_1}{2}\left[\frac{2d_1 - (d_2 - d_1 x)}{\sqrt{4d_1^2 x + (d_2 - d_1 x)^2}} - 1\,\right].
    \end{split}\end{equation*}
    It follows that
    \begin{equation*}\begin{split}
        \sgn\!\left(\frac{\partial\rho_0(d_1,d_2,x)}{\partial x}\right)
        &= \sgn\!\left( 2d_1 - (d_2 - d_1 x) - \sqrt{4d_1^2 x + (d_2 - d_1 x)^2}\right)
    \end{split}\end{equation*}
    If $2d_1 - (d_2 - d_1 x) < 0$, then as $x \geq 0$, it implies that $d_1 < d_2$ in agreement with the negative sign.
    Otherwise, $2d_1 - (d_2 - d_1 x) \ge 0$ and
    \[\begin{split}
            \sgn\!\left(\frac{\partial\rho_0(d_1,d_2,x)}{\partial x}\right)
            &= \sgn\!\Bigl( (2d_1 - (d_2 - d_1 x))^2 - 4d_1^2 x - (d_2 - d_1 x)^2\Bigr) \\
            &= \sgn\left( 4d_1^2 -4d_1(d_2-d_1x)+(d_2-d_1x)^2 -4d_1^2x-(d_2-d_1x)^2\right)\\
            &= \sgn\!\Bigl( 4d_1^2 - 4d_1 d_2\Bigr)\\
            & = \sgn( d_1 - d_2),
    \end{split}\]
    as claimed.
\end{proof}

\BiregularMatrixTheorem*

\begin{proof}
    We prove both the spectral radius formula and the structural characterization by induction on the number of vertices $n = n_1 + n_2$. The base case $n=1$ is trivial. 
    
    If $d_1=0$, or $n_1=0$, then the induced graph on the degree $d_2$ vertices is trivially a $d_2$-regular graph, and the induced graph on the degree $d_1$ vertices is empty. Hence $\rho(M)=d_2=\rho_0(d_1,n_1,d_2,n_2)$, and the theorem holds.
    
    Assume $n > 1$ and $d_1, n_1 > 0$. We first consider the case where $M$ is irreducible. 
    By Corollary~\ref{corollary:x_fixed_on_same_sum_rows}, the Perron eigenvector $x$ is constant on $V_1$ and on $V_2$, with respective values $\chi_1 \neq \chi_2$. 
    Therefore, by Lemma~\ref{lemma:generalized_perturbation}, at least one of the induced graphs $M_{|V_1}$ or $M_{|V_2}$ is empty.
    If the induced graph on $V_2$ is empty, then all $d_2n_2$ edges incident with $V_2$ must connect to $V_1$, implying that $d_1n_1\ge d_2n_2$.
    In conjunction with the assumption that $d_1n_1\le d_2n_2$ it implies the equality $d_1n_1 = d_2n_2$, and by an edge counting argument it implies that the induced graph on $V_1$ is empty as well.
    Therefore, the conditions imply that the induced graph on $V_1$ must be empty.
    
    Next, we establish the regularity of $M_{|V_2}$. For a vertex $v \in V_2$, let $\delta_v = \sum_{v' \in V_2} M_{vv'}$ be its degree within $V_2$. The eigenvector equation $\rho x=Mx$ for row $v$ amounts to:
    \[ \rho \chi_2 = \chi_1 (d_2 - \delta_v) + \chi_2 \delta_v. \]
    As $\chi_1 \neq \chi_2$, solving this equation for $\delta_v$ yields a single solution independent of $v$. Thus, $M_{|V_2}$ is uniformly $\delta$-regular.
    
    Because $M_{|V_1}$ is empty, we write the $n\times n$ symmetric block matrix $M$ as
    \[
        M =\begin{pmatrix}
            0_{n_1 \times n_1} & M_{12} \\
            M_{21}             & M_{22}
        \end{pmatrix}.
    \]
    Let $m_{12}$, $m_{21}$ and $m_{22}$ be the average row sum of the matrices $M_{12}$, $M_{21}$ and $M_{22}$.
    Clearly $m_{12} = d_1$ and as $M_{21} = (M_{12})^T$, we also have $m_{21}=m_{12}n_1/n_2 = d_1 n_1 / n_2$ and $m_{22} = d_2 - d_1 n_1 / n_2$.
    Taking the average of all equations in the same block reduces the problem to a $2 \times 2$ matrix $M'$: 
    \begin{equation*}
        M' \begin{pmatrix}\chi_1\\\chi_2\end{pmatrix} = \rho(M)\begin{pmatrix}\chi_1\\\chi_2\end{pmatrix}, \quad\text{ where } 
        M' = \begin{pmatrix} 0 & d_1 \\ d_1 \frac{n_1}{n_2} & d_2 - d_1 \frac{n_1}{n_2}\end{pmatrix},
    \end{equation*}
    implying that $\rho(M) = \rho(M')$. Solving the quadratic equation for the eigenvalues of $M'$ yields $\rho(M) = \rho_0(d_1,n_1,d_2,n_2)$.
    
    Finally, consider the case where $M$ is reducible, decomposing into irreducible components $M = M_1 \oplus \cdots \oplus M_k$ for $k>1$.
    We have $\rho(M)=\max_i \rho(M_i)$. In each component $M_i$, denote the parameters by $n^{(i)}_1$, $n^{(i)}_2$, and $\nu_i=\frac{n^{(i)}_2}{n^{(i)}_1+n^{(i)}_2}$, so that $\nu = \sum_{i=1}^k \frac{n^{(i)}_1+n^{(i)}_2}{n} \nu_i$.
    If not all $\nu_i$ are equal, there exist indices $i'$ and $i''$ such that $\nu_{i'}<\nu<\nu_{i''}$.
    By the induction hypothesis, $\rho(M_i) \ge \rho_0(d_1,d_2,\nu_i)$. As $d_1,d_2\neq 0$ and $d_1\neq d_2$, Lemma~\ref{lemma:rho0_monotone} ensures the function $\rho_0(d_1,d_2,\cdot)$ is strictly monotone.
    It follows that $\rho(M)\ge \max\bigl(\rho_0(d_1,d_2,\nu_{i'}),\rho_0(d_1,d_2,\nu_{i''})\bigr) > \rho_0(d_1,d_2,\nu)$, contradicting Lemma~\ref{lemma:bounded_by_rho0}.
    Hence, all $\nu_i$ must equal $\nu$. By the induction hypothesis, each component $M_i$ achieves $\rho_0$, has an empty induced subgraph on $V_1$, and a regular induced subgraph on $V_2$. Because $M$ is the direct sum of these components, these structural properties hold globally for $M$, concluding the proof.
\end{proof}

\section{The Average Degree Based Lower Bound}

In this section, we prove Theorem~\ref{theorem:Mne}, stating that given any integers $n>0$ and $e \geq 0$ and a matrix $M \in \M_{n,e}$ such that $\rho(M) = \rho(\M_{n,e})$, the matrix $M$ must be in $\M_{d_1,n_1,d_2,n_2}$, 
with $d_1 = \lfloor 2e/n \rfloor$, $d_2 = d_1+1$, $n_2 = 2e \bmod n$ and $n_1 = n-n_2$. 

We define the function $\rho_1$, based on the function $\rho_0$ defined by equation~\eqref{eq:rho0_d1d2nu}. Given the average degree $\overline{d}$, we set $\nu = \overline{d}-\lfloor\overline{d}\rfloor$ and define
\begin{equation}\label{eq:rho1_def}
    \rho_1(\overline{d}) 
    = \rho_0(d_1,d_2,\nu) 
    = \rho_0(\lfloor\overline{d}\rfloor, \lfloor\overline{d}\rfloor+1, \overline{d}-\lfloor\overline{d}\rfloor).
\end{equation}
The following two lemmas state that $\rho_1(2e/n)$ provides an upper bound on $\rho(\M_{n,e})$ and that $\rho_1$ is a monotone increasing function. 

\begin{lemma}\label{lemma:bounded_by_rho1}
    $\rho(\M_{n,e}) \le \rho_1(\overline{d}) = \rho_1(2 e / n)$.
\end{lemma}

\begin{proof}
    Immediately follows from the inclusion $\M_{d_1,n_1,d_2,n_2} \subset \M_{n,e}$ and the fact that $\rho_1(\overline{d}) = \rho_0(d_1,d_2,\nu)$ for the appropriate $\nu$.
\end{proof}

\begin{lemma}\label{lemma:rho1_monotone}
    The function $\rho_1(x)$ is strictly increasing for $x \ge 1$, it is the constant one for $0 < x \le 1$, and $\rho_1(0) = 0$.
\end{lemma}

\begin{proof}
    The assertion about the value of $\rho_1$ for $x \in [0,1]$ is a simple substitution into equation~\eqref{eq:rho0_d1d2nu}.
    To prove the strict monotonicity for $x \ge 1$, first observe that the monotonicity of $\rho_0$ in Lemma~\ref{lemma:rho0_monotone} implies that 
    \begin{equation}\label{eq:rho1_lbub}
        \lfloor x \rfloor \le \rho_1(x) \le \lceil x \rceil \textrm{ for any }x \ge 1.
    \end{equation}
    Furthermore, if $x$ is not an integer, then its $\nu$ value is in $(0,1)$, so both inequalities are strict. 
    
    Given two numbers $1 \le x_1 < x_2$, we distinguish two cases.
    If $\lfloor x_1 \rfloor = \lfloor x_2 \rfloor$, then $x_1$ and $x_2$ have the same values for $d_1$ and $d_2$ in \eqref{eq:rho1_def}.
    As $x_1$ has a strictly smaller $\nu$ value than $x_2$, the monotonicity of $\rho_0$ implies that $\rho_1(x_1) < \rho_1(x_2)$.

    Otherwise, $\lfloor x_1 \rfloor < \lfloor x_2 \rfloor$, which in conjunction with \eqref{eq:rho1_lbub} yield that:
    \begin{equation}\label{eq:rho1_lbub_x1x2}
        \rho_1(x_1) \leq \lceil x_1 \rceil \leq \lfloor x_2 \rfloor \le \rho_1(x_2).
    \end{equation}
    If at least one of the numbers $x_1, x_2$ is not an integer, then we have at least one strict inequality in \eqref{eq:rho1_lbub_x1x2}
    and the claim follows. On the other hand, if both $x_1,x_2$ are integral then $\rho_1(x_1) = x_1 < x_2 = \rho(x_2)$, concluding the proof.
\end{proof}

We now give several lemmas about the structure of the matrix $M$, when it is also irreducible.

\begin{lemma}\label{lemma:row_sum_diff_le1_cond}
    If $M$ is an irreducible matrix in $\M_{n,e}$ with $\rho(M) = \rho(\M_{n,e})$ and the Perron eigenvector $x$,
    then the degree of any two vertices $i,j$ with $x_i \neq x_j$ differs by exactly one.
\end{lemma}

\begin{proof}
    Given such a matrix $M$, let $d_i := (M\one)_i$ denote the degree of vertex $i$.
By Corollary~\ref{corollary:x_fixed_on_same_sum_rows}, if $x_i \neq x_j$ then necessarily $d_i \neq d_j$.
Suppose, for the sake of contradiction, that there exist vertices $i,j$ such that
\[
x_i \neq x_j \quad\text{and}\quad d_i - d_j \ge 2.
\]

Let $M^{(i,j)} = P^{-1} M P$ where $P$ is the permutation matrix corresponding to the transposition $(i,j)$, and for $0 \le \alpha \le 1$ define
\[
M^{(\alpha)} := (1-\alpha) M + \alpha M^{(i,j)}.
\]
We claim that $M^{(\alpha_0)} \in \mathcal M_{n,e}$, where $\alpha_0 := 1/(d_i-d_j)$.

Since nonnegativity, symmetry, and irreducibility are preserved under convex combinations, it suffices to verify that
$M^{(\alpha_0)} \one \in \mathbb{Z}^n$.
For $k \notin \{i,j\}$, we have $(M^{(i,j)}\one)_k = (M\one)_k = d_k$.
For rows $i,j$ we compute
\[
\begin{aligned}
(M^{(\alpha_0)}\one)_i &= (1-\alpha_0)d_i + \alpha_0 d_j = d_i - 1, \\
(M^{(\alpha_0)}\one)_j &= (1-\alpha_0)d_j + \alpha_0 d_i = d_j + 1.
\end{aligned}
\]
Therefore $M^{(\alpha_0)} \in \mathcal M_{n,e}$, proving the claim.

Since $d_i - d_j \ge 2$, we have $0 < \alpha_0 < 1$.
Thus, Lemma~\ref{lemma:two_matrices_same_Perron} implies that the matrices $M$, $M^{(i,j)}$ share the same Perron eigenvector.
In particular, this implies $x_i = x_j$, contradicting our assumption.
Therefore, no such pair $i,j$ can exist.
\end{proof}

\begin{lemma}\label{lemma:row_sum_diff_le1}
    If $M$ is an irreducible matrix with $\rho(M) = \rho(\M_{n,e})$ then $\dmax - \dmin \le 1$.
\end{lemma}

\begin{proof}
    Given such a matrix $M$, we start with the following claim:

    \begin{claim}
        If $M$ is an irreducible matrix in $\M_{n,e}$ with $\rho(M) = \rho(\M_{n,e})$, and $\dmin, \dmax$ are its minimal and maximal vertex degrees,
        then $\dmax - \dmin \le 2$.
        Furthermore, if $\dmax - \dmin = 2$, 
        the vertices can be partitioned into three non-empty sets $[n] = V_0 \sqcup V_1 \sqcup V_2$, where $V_r$ is the set of degree $\dmin+r$ vertices,
        and the Perron eigenvector $x$ of $M$ is $\chi_0$ on $V_0 \cup V_2$ 
        and $\chi_1$ on $V_1$ for some scalars $\chi_0 \neq \chi_1$.
    \end{claim}
    \begin{claimproof}
        Given such a matrix $M$ with Perron eigenvector $x$, we partition the vertices by degree: $[n] = V_0 \sqcup V_1 \sqcup \cdots \sqcup V_l$,
        so that $v \in V_i$ if its degree is $d_i$, where $d_0 < d_1 < \cdots < d_l$.
        By Corollary~\ref{corollary:x_fixed_on_same_sum_rows}, $x$ is constant on each of the sets $V_i$, where $x|_{V_i} = \chi_i \one_{V_i}$ for some positive constants $\chi_i$.
        If the Perron eigenvector $x$ is a scalar multiple of $\one$, then $Mx = \rhomin x$ implies that all vertices have the same degree, and the claim follows.
        Otherwise, if $\chi_0 \neq \chi_l$ then $d_l-d_0 = 1$ by Lemma~\ref{lemma:row_sum_diff_le1_cond}, and again the claim follows.
        It remains to prove the claim for the case where $\chi_0 = \chi_l$ and there is some index $i$ so that $\chi_i \neq \chi_0$.
        Employing Lemma~\ref{lemma:row_sum_diff_le1_cond} twice yields that 
        $d_i - d_0= d_l - d_i = 1$ and hence $d_l - d_0 = 2$ as claimed.
    \end{claimproof}
    
    By the claim, $\dmax - \dmin \le 2$. Assume by negation that $\dmax - \dmin = 2$.
    Then, by the same claim, the vertices can be partitioned into three non-empty sets $[n] = V_0 \sqcup V_1 \sqcup V_2$, 
    so that $V_r$ is the set of vertices with degree $\dmin+r$.
    Furthermore, the Perron eigenvector $x$ is $\chi_0$ on $V_0 \cup V_2$ and $\chi_1 \neq \chi_0$ on $V_1$.
    Therefore, by Lemma~\ref{lemma:generalized_perturbation}, at least one of the induced graphs $M_{|V_0\cup V_2}$, $M_{|V_1}$ is empty. We first assume that $M_{|V_1}$ is an empty graph.  
    Given vertices $i \in V_0$, $j \in V_1$ and $k \in V_2$, the equations for $x$ being the Perron eigenvector of $M$ with eigenvalue $\rho$ are:
    \begin{align*}
        \rho \cdot \chi_0 &= (Mx)_i = \dmin \cdot\chi_0 + r \cdot (\chi_1 - \chi_0)          \\
        \rho \cdot \chi_1 &= (Mx)_j = (\dmin + 1)\cdot\chi_0                                 \\
        \rho \cdot \chi_0 &= (Mx)_k = (\dmin + 2)\cdot\chi_0 + s \cdot (\chi_1 - \chi_0),    \\
    \end{align*}
    where $r = (M \one_{V_1})_i$ and $s = (M \one_{V_1})_k$.
    Summing up the first and third equations and substituting $\rho \cdot\chi_1$ instead of 
    $(\dmin + 1) \cdot \chi_0$ yields:
    \begin{align*}
        2\rho \chi_0    &= 2(\dmin + 1) \chi_0 + (r + s) (\chi_1 - \chi_0) 
                         = 2 \rho \chi_1 + (r + s) (\chi_1 - \chi_0) \\
        2 \rho (\chi_0-\chi_1) &= (r + s) (\chi_1 - \chi_0) \\
        - 2 \rho        &= r+s,
    \end{align*}
    which is impossible, as $r,s \ge 0$ and $\rho>0$.

    The remaining case is that $M_{|V_0 \cup V_2}$ is an empty graph. 
    Given vertices $i \in V_0$, $j \in V_1$ and $k \in V_2$, the equation for $x$ being the Perron eigenvector of $M$ yields:
    \begin{align*}
        \rho \cdot \chi_0 &= (Mx)_i = \dmin \cdot \chi_1\\
        \rho \cdot \chi_0 &= (Mx)_k = (\dmin + 2) \cdot \chi_1,
    \end{align*}
    implying that $2\chi_1 = 0$. As the Perron eigenvector is strictly positive, this is a contradiction, concluding the proof of the lemma.
\end{proof}

\MneTheorem*

\begin{proof}
    We prove the claim by induction on $n$, where the base case for $n=1$ is trivial.

    Suppose that $n>1$ and assume the statement is true for all smaller values of $n$. 
    Let $M$ be a matrix that minimizes the spectral radius $\rho$ on $\M_{n,e}$. 
    Consider first the case where matrix $M$ has a non-trivial decomposition into a direct sum of smaller matrices, $M = M_1 \oplus \cdots \oplus M_k$ for $k>1$. 
    Let $\overline{d_i}$ denote the average degree of the matrix $M_i$.
    Then, using Lemma~\ref{lemma:bounded_by_rho1} to upper bound $\rho(M)$, Lemma~\ref{lemma:rho1_monotone} for the monotonicity of $\rho_1$, and applying the induction hypothesis to each connected component separately, we get:
    \[
        \rho_1(\overline{d}) 
        \ge \rho(M) 
        = \max_i \rho(M_i) 
        = \max_i \rho_1(\overline{d_i}) \ge \rho_1(\max_i \overline{d_i}) \ge \rho_1(\overline{d}).
    \]
    Therefore, we have equality $\rho(M) = \rho_1(\overline{d})$.
    Moreover, if $\overline{d} \ge 1$, as $\rho_1(x)$ has strict monotonicity for $x \ge 1$, then $\max_i \overline{d_i} = \overline{d}$, so all $M_i$ have average degree $\overline{d}$.
    Therefore, by induction, each $M_i$ may only have vertex degrees $d_1,d_2$, and so does $M$, implying that $M \in \M_{d_1,n_1,d_2,n_2}$. 
    Otherwise, $\overline{d} < 1$, so $\overline{d_i} \le 1$ for all $i$. Again, by induction, this implies that vertex degrees are either $d_1=0$ or $d_2=1$ for each $M_i$ and thus for $M$, implying the required result as well.

    It remains to handle the case where $M$ is irreducible. 
    By Lemma~\ref{lemma:row_sum_diff_le1}, the difference between the minimal and maximal degrees is at most one. 
    Therefore, $M \in \M_{d_1,n_1,d_2,n_2}$ with $d_1,n_1,d_2,n_2$ as above. 
\end{proof}

\section{Scaling Law For A Large Average Degree}
Recall that the lower bounds on $\rho$ in both Theorem~\ref{theorem:Mn1d1n2d2} and Theorem~\ref{theorem:Mne} are given by the function $\rho_0$,
defined by \eqref{eq:rho0_d1d2nu}, that depends only on $d_1, d_2$ and the ratio $\nu = n_2/(n_1+n_2)$. 
In light of Figure~\ref{fig:rho_opt} (b), it is natural to ask about the asymptotic behavior of $\rho_0(d_1,d_2,\nu)$ when the average degree is large. More precisely, for fixed numbers $\Delta>0$ and $\nu \in [0,1)$, 
we would like to study the function $\rho_0(d,d+\Delta,\nu)$ when $d$ is large, and with $\Delta=1$ for the $\M_{n,e}$ case. 
The following theorem gives a precise answer to this question.
\begin{theorem}\label{theorem:rhoopt_residue}
    Given $\Delta>0$, $\nu \in [0,1)$ and $d \ge 0$, define the function $f(d,\nu,\Delta)=\rho_0(d,d+\Delta,\nu)-\overline{d}$, where $\overline{d} = d + \nu \Delta$,
    then: 
    \begin{equation*}
        \lim_{d \rightarrow \infty} \bigl(d \cdot f(d,\nu,\Delta)\bigr) = \Delta^2 \nu(1-\nu) \cdot \max(\nu,1-\nu).
    \end{equation*}
\end{theorem}

\begin{proof}
    Given $d \ge 0$, $\Delta > 0$ and $0 \le \nu \le 1$, the average degree is $\overline{d} = d+\nu \Delta$. 
    We rewrite equation~\eqref{eq:rho0_d1d2nu} for $\rho_0$ where $d_1 = d$ and $d_2 = d+\Delta$, as follows:
    \begin{align}
        \rho_0(d,d+\Delta,\nu)
        = \begin{cases} \scriptstyle
                \frac{1}{2}\left[d+\Delta - d \frac{1-\nu}{\nu} + \sqrt{4d^2 \frac{1-\nu}{\nu} + (d+\Delta - d \frac{1-\nu}{\nu})^2}\,\right]        & \textrm{if } \nu \ge \frac{d}{2d+\Delta} \\
                \addlinespace[0.1cm] \scriptstyle
                \frac{1}{2}\left[d - (d+\Delta) \frac{\nu}{1-\nu} + \sqrt{4(d+\Delta)^2\frac{\nu}{1-\nu} + (d - (d+\Delta) \frac{\nu}{1-\nu})^2}\,\right] & \textrm{if } \nu < \frac{d}{2d+\Delta}
            \end{cases} \label{eq:rhoopt_simple}
   \end{align}
    Plugging this into the formula for $f$ and simplifying yields:
   \begin{align*}
        d \cdot f(d,\nu,\Delta) &=& \begin{cases} \scriptstyle
                \frac{d}{2}\left[(1 - 2\nu)\Delta - \frac{d}{\nu} + \frac{d}{\nu}\sqrt{1 + \frac{2\nu(2\nu-1)\Delta}{d} + \frac{\nu^2\Delta^2}{d^2}}\,\right] & \textrm{if } \nu \ge \frac{d}{2d+\Delta}  \\
                \addlinespace[0.1cm] \scriptstyle
                \frac{d}{2}\left[- \Delta \frac{(3-2\nu)\nu}{1-\nu} - \frac{d}{1-\nu}  + \frac{d}{1-\nu}\sqrt{1+\frac{2(3-2\nu)\nu\Delta}{d}+\frac{(4-3\nu)\nu\Delta^2}{d^2}}\,\right]     & \textrm{if } \nu < \frac{d}{2d+\Delta}
            \end{cases}
    \end{align*}
    Using the Taylor expansion $\sqrt{1+z} = 1 + \frac{z}{2} - \frac{z^2}{8} + O(z^3)$ yields:
   \begin{align*}
        d \cdot f(d,\nu,\Delta) 
        &= \begin{cases} \scriptstyle
                \frac{d}{2}\left[(1 - 2\nu)\Delta - \frac{d}{\nu} + \frac{d}{\nu}(1+\frac{\nu(2\nu-1)\Delta}{d}+\frac{2\nu^3(1-\nu)\Delta^2}{d^2}+O(\frac{1}{d^3}))\,\right]  & \textrm{if } \nu \ge \frac{d}{2d+\Delta}  \\
                \addlinespace[0.1cm] \scriptstyle
                \frac{d}{2}\left[-\Delta \frac{(3-2\nu)\nu}{1-\nu} - \frac{d}{1-\nu}  + \frac{d}{1-\nu}(1+\frac{(3-2\nu)\nu\Delta}{d}+\frac{2(1-\nu)^3\nu\Delta^2}{d^2}+O(\frac{1}{d^3}))\,\right] & \textrm{if } \nu < \frac{d}{2d+\Delta}
            \end{cases} \\
        &= \begin{cases}
                \nu^2(1-\nu)\Delta^2 + O(\frac{1}{d})\,  & \textrm{if } \nu \ge \frac{d}{2d+\Delta}  \\
                \addlinespace[0.1cm]
                (1-\nu)^2 \nu\Delta^2 + O(\frac{1}{d})   & \textrm{if } \nu < \frac{d}{2d+\Delta}
            \end{cases}            
    \end{align*}
    Taking the limit $d \rightarrow \infty$ yields $\nu^2(1-\nu)\Delta^2$ if $\nu > \frac{1}{2}$ and $(1-\nu)^2\nu\Delta^2$ for $\nu \le \frac{1}{2}$, and the Theorem is proved.
\end{proof}

\section{Discrete Graphs}\label{sect:disrete_graphs}

In this section, we establish a necessary and sufficient condition on $d_1, n_1, d_2, n_2$ to guarantee the existence of a simple graph $G \in \mathcal{G}_{d_1, n_1, d_2, n_2}$ such that $\rho(G) = \rho(\mathcal{M}_{d_1, n_1, d_2, n_2})$. Furthermore, we prove the analogue, Theorem~\ref{theorem:when_biregular_simple_connected_graph}, which imposes the additional requirement that $G$ is connected.

We start with the proof of Theorem~\ref{theorem:when_biregular_simple_graph}, which allows us to characterize the pairs $(n, e)$ for which there exists a simple graph $G$ satisfying $\rho(G) = \rho(\mathcal{M}_{n, e})$. By Theorem~\ref{theorem:Mne}, such $\rho$-minimizing graphs must belong to $\mathcal{G}_{d_1, n_1, d_2, n_2}$, where the parameters are determined by $n$ and $e$ with $d_2 = d_1 + 1$. Consequently, $(n, e)$ pairs for which $\rho(\mathcal{M}_{n, e}) = \rho(\mathcal{G}_{n, e})$ correspond to cases in which Hong's conjecture can be answered affirmatively.

We then employ Theorem~\ref{theorem:when_biregular_simple_graph} to deduce two primary consequences regarding the existence of such $(n, e)$ pairs. First, within this section, Theorem~\ref{theorem:nu_simple_graph} shows that the ratio $\nu=n_2/n$ must belong to a discrete set of values determined by $d_1$ and $d_2$. Second, in Section~\ref{section:num_edges_simple_graph}, we apply these structural conditions to prove Theorem~\ref{theorem:num_edges_simple_graph}, demonstrating that for a fixed number of vertices $n$, the number of non-trivial values for $e$ is at least linear in $n$.

\subsection{Auxiliary Realization Lemmas}

Following are two auxiliary lemmas used in the proof of Theorem~\ref{theorem:when_biregular_simple_graph}. The lemmas give necessary and sufficient conditions for the existence of simple graphs with particular specifications. 
Their proofs are based on stronger versions of the Erd\H{o}s-Gallai 1960 and of the Gale-Ryser 1957 conditions.

\begin{lemma}\label{lemma:simple_connected_regular}
    Given a positive integer $n$ and a non-negative integer $d$ so that $nd$ is even and $d \le n-1$, there exists a $d$-regular simple graph with $n$ vertices.
\end{lemma}

\begin{lemma}\label{lemma:simple_bipartite}
    Given positive integers $n_1, n_2$ and non-negative integers $d_1, d_2$, satisfying $n_1 d_1 = n_2 d_2$, there exists a simple bipartite graph where one side has $n_1$ degree $d_1$ vertices and the other side has $n_2$ degree $d_2$ vertices, if and only if $n_1 \geq d_2$, or equivalently $n_2 \geq d_1$.
\end{lemma}

\noindent
The two lemmas are immediate consequences of the following:

\begin{proposition}[Tripathi and Vijay~\cite{tripathi2003note}]\label{proposition:graphic_sequence}
    The non-negative sequence of integers $a_1 \ge a_2 \ge \cdots \ge a_n$ is graphic if and only if $\sum_{i=1}^n a_i$ is even and 
    the following inequality holds for all $k\in\{1,\ldots,n-1\}$ with $a_k > a_{k+1}$ and for $k=n$:
    \begin{equation}\label{eq:erdos_gallai}
        \sum_{i=1}^k a_i \le k(k-1) + \sum_{i=k+1}^n \min(a_i,k).
    \end{equation}
\end{proposition}

\begin{proposition}[Theorem 2 from Berger~\cite{berger2014note}]\label{proposition:bipartite_realization}
    Given non-negative integers $a_1 \ge a_2 \ge \cdots \ge a_n$ and $b_1, b_2, \ldots, b_n$, satisfying $\sum_{i=1}^n a_i = \sum_{i=1}^n b_i$,
    there exists a simple bipartite graph with $n$ vertices on each side, where the left degrees are 
    $\{a_i\}$ and the right degrees are $\{b_i\}$, if and only if the following inequality holds for all $k\in\{1,\ldots,n-1\}$ with $a_k > a_{k+1}$ and for $k=n$:
    \begin{equation}\label{eq:berger}
        \sum_{i=1}^n \min(b_i,k) \ge \sum_{i=1}^k a_i.
    \end{equation}
\end{proposition}

\begin{proof}[Proof of Lemma~\ref{lemma:simple_connected_regular}]
    Given non-negative integers $n,d$, by Proposition~\ref{proposition:graphic_sequence}, there exists a $d$-regular simple graph with $n$ vertices, 
    if and only if $dn$ is even and \eqref{eq:erdos_gallai} holds for $k=n$, namely that $n d \le n(n-1)$.
\end{proof}

\begin{proof}[Proof of Lemma~\ref{lemma:simple_bipartite}]
    Given $n_1, d_1, n_2, d_2 > 0$ such that $n_1 d_1 = n_2 d_2$, 
    assume without loss of generality that $d_1 \le d_2$, which implies $n_1 \ge n_2$. 
    To show the existence of a simple bipartite bi-regular graph, we invoke Proposition~\ref{proposition:bipartite_realization} with: $n=n_1$; $a_i=d_1$ for all $i$; and $b_i$ being $d_2$ for $i \leq n_2$ and zero otherwise. 
    By the proposition, it suffices to check that \eqref{eq:berger} holds for $k = n_1$, namely that $n_2 \min(d_2,n_1) \ge d_1 n_1$, which simplifies to the lemma's condition $n_1 \geq d_2$.    
\end{proof}

\subsection{Existence of Simple Graphs}

\WhenBiregularSimpleGraphTheorem*

\begin{proof}
    Let $d_1, n_1, d_2, n_2$ be integers satisfying the conditions of the theorem.
    For a graph $G \in \G_{d_1,n_1,d_2,n_2}$, let $V_i$ denote the set of degree $d_i$ vertices for $i=1,2$.
    Then by Theorem~\ref{theorem:Mn1d1n2d2}, we have $\rho(G) = \rho(\M_{d_1,n_1,d_2,n_2})$ if and only if $G[V_1]$ is empty and $G[V_2]$ is regular. 
    This is equivalent to $G$ being the union of two graphs: 
    \begin{enumerate}
        \item $G_1$, which is a regular graph on $V_2$ with degree $d'_2 = (n_2 d_2 - n_1 d_1)/n_2$,
        \item $G_2$, which is a bipartite bi-regular graph, with sides $V_1, V_2$ and respective degrees of $d_1$ and $n_1 d_1 / n_2$.
    \end{enumerate}
    Necessary and sufficient conditions for the existence of these simple graphs $G_1, G_2$, on top of the condition that degrees are integral, which is theorem condition (ii), are:
    \begin{description}
        \item[$G_1$ - ] by Lemma~\ref{lemma:simple_connected_regular}, the conditions are that $d'_2 n_2$ is even and that $d'_2 \le n_2-1$, equivalent to theorem conditions (i) and (iv).
        \item[$G_2$ -] by Lemma~\ref{lemma:simple_bipartite}, the condition is $n_2 \ge d_1$, which is theorem condition (iii).
    \end{description}
\end{proof}

\NuSimpleGraphTheorem*
\begin{proof}
    Given $d_1,d_2$ Theorem~\ref{theorem:when_biregular_simple_graph} gives necessary and sufficient conditions for $\rho(\G_{d_1,n_1,d_2,n_2})=\rho(\M_{d_1,n_1,d_2,n_2})$ to hold for two positive integers $n_1,n_2$ such that $n_1 d_1 \le n_2 d_2$. 
    Let $k$ be an integer such that $k\nu$ is integral. Then, the requirement that $n_1=(1-\nu)n$ and $n_2=\nu n$ be integral holds if $n \in k\N$.
    Theorem~\ref{theorem:when_biregular_simple_graph} condition (ii) requires that $n_2$ divides $n_1 d_1$, which is equivalent to
    \[
        \frac{n_1 d_1}{n_2} = i \in \{1,\ldots,d_2\}.
    \]
    Rewriting yields $\frac{1-\nu}{\nu} d_1 = i$, or $\nu = \frac{d_1}{d_1+i}$.
    The other conditions of the theorem are as follows:
    \begin{description}
        \item[(i)] $n_2 d_2 + n_1 d_1 = (\nu d_2 + (1-\nu) d_1) n$, is an even integer if $n \in 2 k \N$,
        \item[(iii)] $n_2 \ge d_1$, which holds for a sufficiently large $n$,
        \item[(iv)] $n_2(n_2-1) \ge n_2 d_2 - n_1 d_1$, holding as well for a sufficiently large $n$.
    \end{description}
    It remains to consider the case where $n_1 d_1 > n_2 d_2$, in which Theorem~\ref{theorem:when_biregular_simple_graph} applies by exchanging $n_1,d_1$ with $n_2, d_2$. Condition (ii) requires that $n_2 d_2 / n_1$ to be some integer $i$, where $i \in \{1,\ldots,d_1-1\}$. 
    This is equivalent to $\nu d_2 = (1-\nu) i$, and isolating for $\nu$ we get $\nu = \frac{i}{d_2+i}$. The other conditions of Theorem~\ref{theorem:when_biregular_simple_graph} are analyzed in a similar manner and hold for a sufficiently large $n \in 2k\N$.
\end{proof}

We illustrate Theorem~\ref{theorem:when_biregular_simple_graph} by Figure~\ref{fig:GraphBounds23} for the 2-3 bi-regular case, and Figure~\ref{fig:GraphBounds34} for the 3-4 bi-regular case. 

\begin{figure}[h!]
    \centering
    \includegraphics[width=0.8\textwidth]{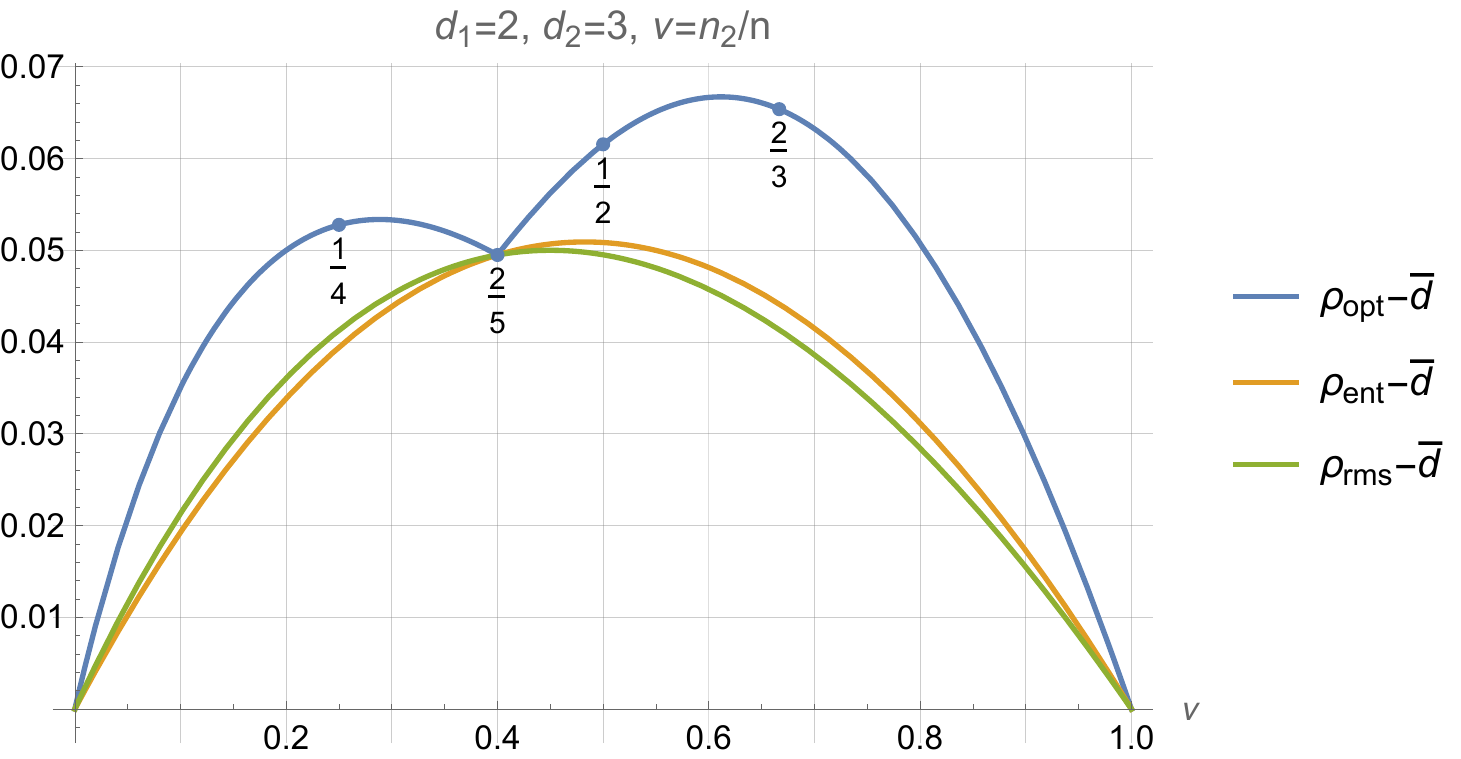}
    \includegraphics[width=1.0\textwidth]{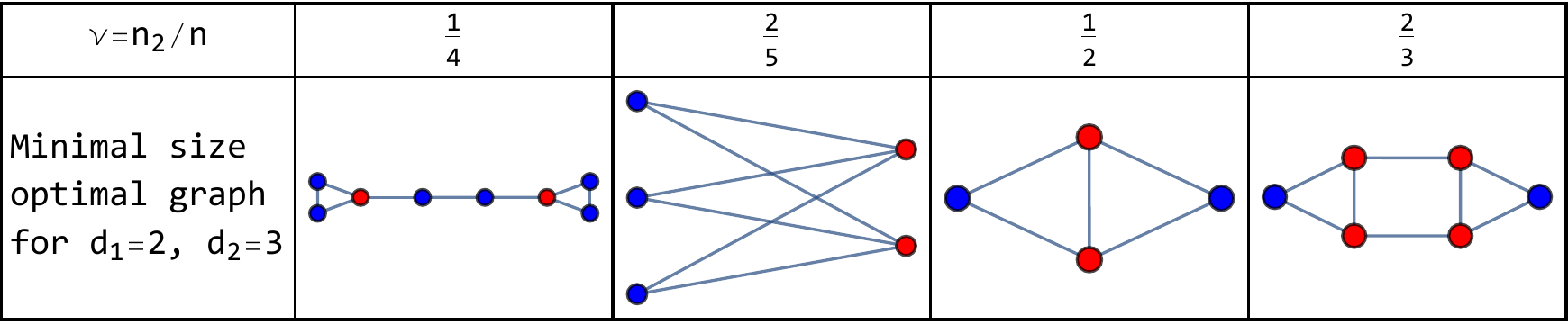}
    \caption{
        Comparing the three lower bounds for bi-regular graphs with $d_1=2, d_2=3$, when varying $\nu=n_2/(n_1+n_2)$.
        To facilitate comparison, the average degree baseline $\overline{d}$ is subtracted.
        The plots are: $\rho(\M_{d_1,n_1,d_2,n_2})-\overline{d}$ in blue, $\rhoent(d_1,n_1,d_2,n_2)-\overline{d}$ in orange and $\rhorms(d_1,n_1,d_2,n_2)-\overline{d}$ in green,
        where $\rhoent$ and $\rhorms$ are the entropy and root mean square lower bounds defined by Propositions~\ref{proposition:entropy_lb},~\ref{prop:deg_rms_lb}
        Additionally, we mark the $\nu$ values $\largeNu_{2,3}=\{\frac{1}{4}, \frac{2}{5}, \frac{1}{2}, \frac{2}{3}\}$ defined by \eqref{eq:nu_set} and give a minimal size graph with $\rho(G)=\rho(\M_{d_1,n_1,d_2,n_2})$, for each one.
    }
    \label{fig:GraphBounds23}
\end{figure}

\begin{figure}[h!]
    \centering
    \includegraphics[width=0.8\textwidth]{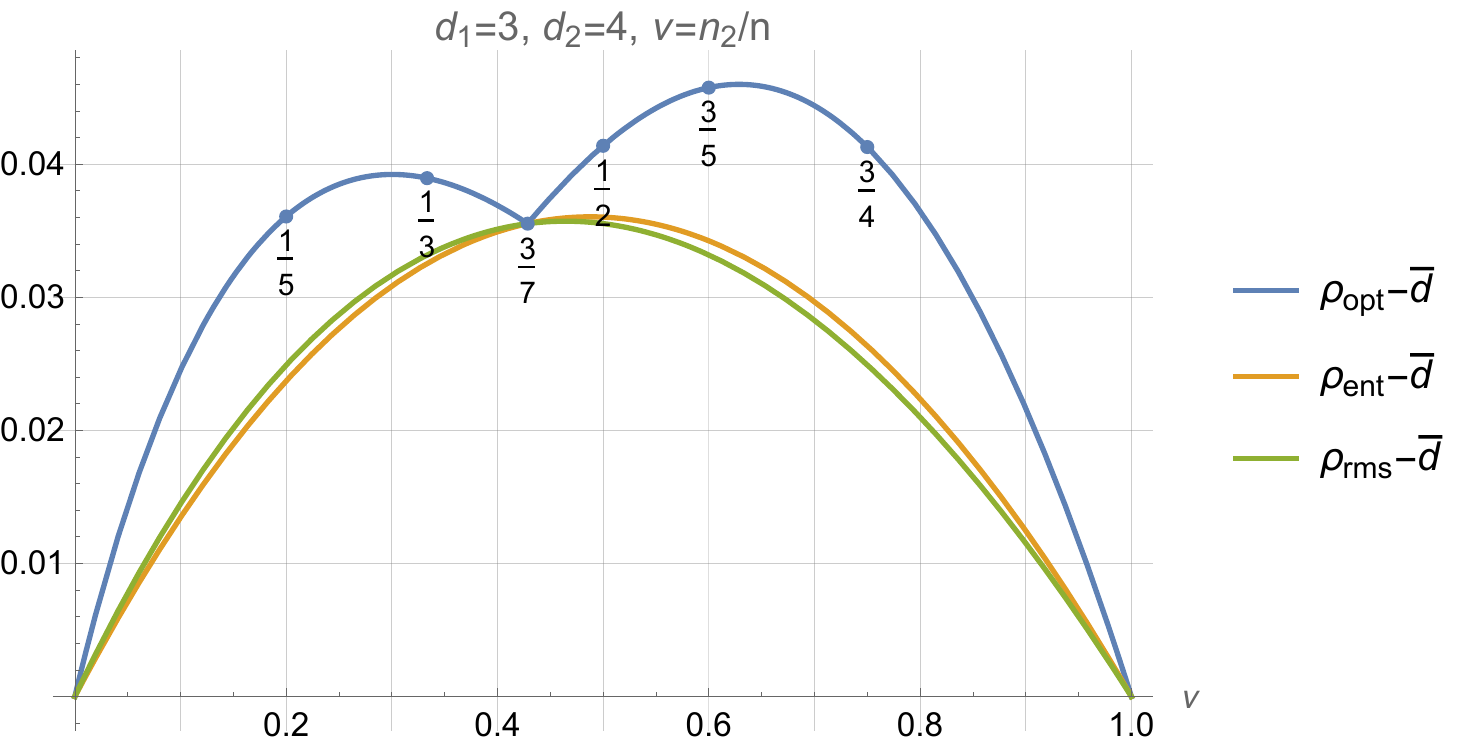}
    \includegraphics[width=1.0\textwidth]{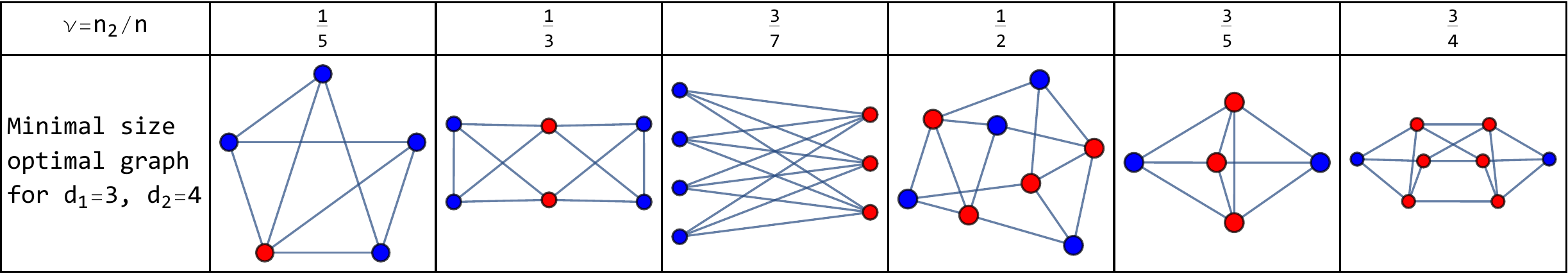}
    \caption{Similar to Figure~\ref{fig:GraphBounds23} with $d_1=3, d_2=4$ and $\largeNu_{3,4}=\{\frac{1}{5}, \frac{1}{3}, \frac{3}{7}, \frac{1}{2}, \frac{3}{5}, \frac{3}{4}\}$.}
    \label{fig:GraphBounds34}
\end{figure}

\subsection{Connectedness Constraints}

\begin{theorem}\label{theorem:when_biregular_simple_connected_graph}
    Let $n_1, n_2$ be positive integers and $d_1, d_2$ be non-negative integers such that $d_1 \neq d_2$ and $n_2 d_2 \ge n_1 d_1$,
    then a necessary and sufficient condition for the existence of a connected graph in $\G_{d_1,n_1,d_2,n_2}$ with $\rho(G)=\rho(\M_{d_1,n_1,d_2,n_2})$, is that, in addition to the four conditions imposed by Theorem~\ref{theorem:when_biregular_simple_graph}, at least one of the following conditions hold:
    \begin{enumerate}[label=(\alph*)]
        \item $d_1, d_2 \ge 2$,
        \item $n_1 = d_2 = 1$, $n_2=d_1 > 1$: $G$ is a star with the central vertex in $V_1$,
        \item $n_2 = d_1 = 1$, $n_1 = d_2 > 1$: $G$ is a star with the central vertex in $V_2$,
        \item $d_1 = 1$, $n_2 = 2$, $d_2 > 1$, $n_1 = 2(d_2-1)$: $G$ is a double star with the two centers in $V_2$,
        \item $d_1 = 1$, $d_2 > 1$ and $n_2 d_2 - n_1 d_1 \ge 2 n_2$: the induced graph $G[V_2]$ is a regular graph of degree at least two, where each $V_2$ vertex is connected to $n_1/n_2$ leaves in $V_1$. 
    \end{enumerate}
\end{theorem}

\begin{proof}
    Given $n_1, n_2, d_1$ and $d_2$ satisfying the requirements of Theorem~\ref{theorem:when_biregular_simple_graph}, let $\rho_0 = \rho(\M_{d_1,n_1,d_2,n_2})$, and let $G$ be a simple graph in $\G_{d_1,n_1,d_2,n_2}$ such that $\rho(G) = \rho_0$. 
    As connectedness implies that $d_1$ and $d_2$ are non-zero, we consider separately three cases: $d_1,d_2\ge 2$; $d_2 = 1 < d_1$; and $d_1 = 1 < d_2$.

    \begin{claim}
        The theorem holds if $d_1,d_2\ge 2$, which is case (a) of the theorem.
    \end{claim}
    \begin{claimproof}
        If $G$ is connected, we are done. Otherwise, we show that $G$ can be made connected via edge switching, while maintaining the condition $\rho(G) = \rho_0$. 
        As the minimal degree is at least two, every connected component has a cycle, which either (i) has a bipartite edge, or (ii) lies entirely in $G[V_2]$.
        If there are cycles $C_1, C_2$ in distinct components, that are of the same type, (i) or (ii), 
        we choose $e_1=(u_1,v_1) \in C_1$ and $e_2=(u_2,v_2) \in C_2$ such that $v_1,v_2 \in V_2$ and $u_1,u_2 \in V_i$ for the same $i\in\{1,2\}$.
        Let $G'$ be the graph obtained by switching the edges, $G'=G \setminus \{e_1,e_2\} \cup \{(u_1,v_2),(u_2,v_1)\}$. 
        Then $G'$ is a graph in $\G_{d_1,n_1,d_2,n_2}$ such that $G'[V_1]$ is the empty graph and $G'[V_2]$ is a regular graph, so that $\rho(G') = \rho(G)$,
        and where $G'$ has fewer connected components than $G$. 
        Repeating this process either produces a connected graph, or else, it produces a graph with two connected components, so that all cycles in the first are of type (i), while all cycles in the second are of type (ii).
        
        Let $C_1$ and $C_2$ be two cycles in distinct components, $K_1, K_2$ respectively, where $C_1$ is in $G[V_2]$, and $C_2$ has a bipartite edge. Since any $V_2$ vertex has the same number of $V_1$ neighbors, $n_1 d_1 / n_2 > 0$, the component $K_1$ must have a bipartite edge, and we can choose edge $e_1=(u_1,v_1) \in K_1$ and $e_2=(u_2,v_2) \in C_2$ with $u_1,u_2 \in V_1$ and $v_1,v_2 \in V_2$. As before, we define the switch graph $G'=G \setminus \{e_1,e_2\} \cup \{(u_1,v_2),(u_2,v_1)\}$, where we have $G' \in \G_{d_1,n_1,d_2,n_2}$ and $\rho(G') = \rho(G)$. Moreover, we claim that $G'$ is connected. 
        Indeed, in $G'$ there is a path $P$ from $u_2$ to $v_2$ using the remaining edges in the cycle $C_2$, and there is a path from $u_1$ to $v_1$ consisting of $(u_1,v_2)$, the path $P$ and $(u_2,v_1)$, concluding the proof.
    \end{claimproof}

    \begin{claim}
        The theorem holds if $d_2 = 1 < d_1$, under the conditions of case (b).        
    \end{claim}
    \begin{claimproof}
        As for the graph $G$, each $V_2$ vertex has $n_1 d_1 / n_2 > 0$ neighbors in $V_1$ and $d_2=1$, it follows that $G[V_2]$ is empty.
        As $G[V_1]$ is empty as well, the graph $G$ must be bipartite. 
        Therefore, $G$ is a collection of $n_1$ stars, each with $n_2/n_1$ leaves, which is connected if and only if $n_1=1$.
    \end{claimproof}

    \begin{claim}
        The theorem holds if $d_1 = 1 < d_2$, under the conditions of cases (c), (d) or (e).                
    \end{claim}
    \begin{claimproof}
         The degree of the regular graph $G[V_2]$ is
         \[
            d'_2 = d_2 - \frac{d_1 n_1}{n_2} = d_2 - \frac{n_1}{n_2}.
        \]
        If $d'_2=0$ then, as the $V_1$ vertices are leaves that do not contribute to connectivity, we need $n_2=1$ to be connected, which is case (c) of the star graph.
        Otherwise, if $d'_2=1$ then we need $n_2=2$ to be connected, which is case (d) of the double star graph.
        Otherwise, $d'_2 \ge 2$, and by switching edges of $G[V_2]$ if necessary we can make $G$ connected, which is case (e).
    \end{claimproof}
\end{proof}

\section{Counting Simple Graphs with Optimal Spectral Radius}\label{section:num_edges_simple_graph}

This section proves Theorem~\ref{theorem:num_edges_simple_graph}, bounding the function $E(n)$. 
Recall that given the number of vertices is $n$, this function counts the number of edge values $e$ such that $\rho(\G_{n,e}) = \rho(\M_{n,e})$ and $n$ does not divide $2e$. Note that if $n$ does divide $2e$, then the spectral minimizers are regular graphs, which we know well.

\NumEdgesSimpleGraphTheorem*

The proof is a consequence of an exact formula for $E(n)$. However, before stating the formula, we establish the required number-theoretic functions. 
Let $\phi(n)$ denote Euler's totient function. Let $\chi_4$ denote the non-principal Dirichlet character modulo 4, defined strictly for odd integers by its congruence class:
\begin{equation}
    \chi_4(n) = \begin{cases} 
        1 & \text{if } n \equiv 1 \pmod 4 \\ 
        -1 & \text{if } n \equiv 3 \pmod 4 \\
        0 & \text{otherwise}
    \end{cases}
\end{equation}

Recall that the Dirichlet convolution of two arithmetical functions $a(n)$ and $b(n)$ is given by $(a*b)(n) = \sum_{d|n} a(d)b(n/d)$. 
For odd $n$, we let $f(n)$ denote the Dirichlet convolution of $\chi_4$ and $\phi$, defined as:
\begin{equation}
    f(n) = (\chi_4 * \phi)(n) = \sum_{d|n} \chi_4(d)\phi\left(\frac{n}{d}\right),
\end{equation}
and we let
\begin{equation}
    P(n) = \sum_{k=1}^n \gcd(n,k)
\end{equation}
denote Pillai's arithmetical function. 

Then, $E(n)$ is given by the following exact formula:
\begin{theorem}\label{theorem:exact_En}
    For any integer $n \ge 3$, the number of distinct values of $e$ such that $n$ does not divide $2e$ and $\rho(\mathcal{G}_{n,e}) = \rho(\mathcal{M}_{n,e})$ is  
    \[
    E(n) = \begin{cases}
        2P(n/2) - \frac{n}{2} - 1 & \text{if } n \text{ is even} \\
        \frac{3}{4}P(n) + \frac{n - f(n) - 8 + \chi_4(n)}{4} & \text{if } n \text{ is odd}
    \end{cases}
    \]
\end{theorem}

\begin{proof}
    By Theorem~\ref{theorem:Mne}, any spectral radius minimizer in $\M_{n,e}$ corresponds to a bi-regular weighted graph belonging to the family $\M_{d_1,n_1,d_2,n_2}$, where $d_1 = \lfloor 2e/n \rfloor$, $d_2 = d_1+1$, $n_2 = 2e \bmod n$ and $n_1 = n-n_2$.
    Theorem~\ref{theorem:when_biregular_simple_graph} gives a necessary and sufficient condition for the existence of a simple $\rho$-minimizing graph in $\M_{d_1,n_1,d_2,n_2}$. 
    Therefore, given $n$, the number $E(n)$ of valid choices for $e$ is equal to the number of pairs $(n_2, d_1)$ satisfying constraints (i)-(iv) of Theorem~\ref{theorem:when_biregular_simple_graph}, where $1 \le n_2 \le n-1$.
    
    Note that if the dominance condition $n_1 d_1 \le n_2 d_2$ required by Theorem~\ref{theorem:when_biregular_simple_graph} is violated, we exchange the roles of $d_1,n_1$ and $d_2,n_2$ to obtain analogous constraints (i')-(iv'). 
    Using only $n_2,d_1$, we rewrite the dominance condition as:
    \begin{equation}\label{eq:dominance}
        (n - 2 n_2) d_1 \le n_2
    \end{equation}

    Since $n$ does not divide $2e$, we restrict $n_2$ to $1 \le n_2 \le n-1$. Thus, the total number of valid edge values can be expressed as the sum:
    \begin{equation}
        E(n) = \sum_{n_2=1}^{n-1} E(n, n_2),
    \end{equation}
    where $E(n, n_2)$ is the number of valid degree choices $d_1$ for a fixed $n_2$. For the rest of the proof we denote $g = \gcd(n, n_2)$, where $n$ is fixed and $n_2$ is implicit from the context. 
    
    We partition $E(n, n_2)$ into two mutually exclusive sets based on the dominance condition. Let $E_1(n, n_2)$ count the valid $d_1$ values satisfying $n_1 d_1 \le n_2 d_2$, and let $E_2(n, n_2)$ count the valid $d_1$ values satisfying $n_1 d_1 > n_2 d_2$. Thus, $E(n, n_2) = E_1(n, n_2) + E_2(n, n_2)$.

    Preparing for the computation of $E_1$ and $E_2$, we first define the pre-parity candidate pools $C_1, C_2$. Let $C_1(n,n_2)$ be the number of $d_1$ values satisfying $n_1 d_1 \le n_2 d_2$ alongside the structural conditions (ii) and (iii) only. Similarly, let $C_2(n,n_2)$ be the number of $d_1$ values satisfying $n_1 d_1 > n_2 d_2$ alongside the mirrored conditions (ii') and (iii') only.

    \begin{claim}\label{claim:C1_formula}
        \[
            C_1(n,n_2) = \begin{cases}
                g+1 & \text{ if } 2n_2 \ge n \\
                2   & \text{ if } \frac{n-2n_2}{g} = 1 \\
                1   & \text{ if } \frac{n-2n_2}{g} > 1
            \end{cases}
        \]
        
    \end{claim}
    \begin{claimproof}
        The candidates counted by $C_1(n,n_2)$ must satisfy the conditions of Theorem~\ref{theorem:when_biregular_simple_graph}: (ii) $n_2 \mid n d_1$ and (iii) $n_2 \ge d_1$, subject to the dominance condition~\eqref{eq:dominance}. 
        
        From (ii), dividing by $g = \gcd(n,n_2)$, we obtain $\frac{n_2}{g} \mid \frac{n}{g} d_1$. Since $\gcd(\frac{n_2}{g}, \frac{n}{g}) = 1$, we must have $\frac{n_2}{g} \mid d_1$. Thus, there exists an integer $m \ge 0$ such that:
        \begin{equation}\label{eq:d1_m}
            d_1 = m \frac{n_2}{g}.
        \end{equation}
        Substituting~\eqref{eq:d1_m} into (iii) yields $m \frac{n_2}{g} \le n_2$, which simplifies to $m \le g$. Thus, $m \in \{0, 1, \dots, g\}$.
        
        We now apply the dominance condition~\eqref{eq:dominance}. Substituting~\eqref{eq:d1_m} into $(n-2n_2)d_1 \le n_2$ and dividing by $n_2$ gives the additional constraint:
        \begin{equation}\label{eq:m_bound}
            m \frac{n-2n_2}{g} \le 1.
        \end{equation}
        We evaluate the number of valid integers $m$ based on the ratio $\frac{n-2n_2}{g}$:
        \begin{itemize}
            \item If $2n_2 \ge n$, the term $(n-2n_2)$ is non-positive, inequality~\eqref{eq:m_bound} holds unconditionally for all $m$, yielding $g+1$ valid candidates.
            \item If $\frac{n-2n_2}{g} = 1$, inequality~\eqref{eq:m_bound} becomes $m \cdot 1 \le 1$, yielding exactly $2$ valid candidates, $m \in \{0, 1\}$.
            \item If $\frac{n-2n_2}{g} > 1$, inequality~\eqref{eq:m_bound} yields exactly $1$ valid candidate, $m=0$.
        \end{itemize}
    \end{claimproof}

    \begin{claim}\label{claim:C2_formula}
        \[
            C_2(n,n_2) = \begin{cases}
                0   & \text{ if } 2n_2 \ge n \\
                g-1 & \text{ if } \frac{n-2n_2}{g} = 1 \\
                g   & \text{ if } \frac{n-2n_2}{g} > 1
            \end{cases}
        \]

    \end{claim}
    \begin{claimproof}
        The candidates counted by $C_2(n,n_2)$ operate under the strict inverse of the dominance condition $(n-2n_2)d_1 > n_2$, where the mirrored structural constraints are (ii') $(n-n_2) \mid n_2(d_1+1)$ and (iii') $d_1 + n_2 \le n-1$.

        From (ii'), dividing by $g$ yields $\frac{n-n_2}{g} \mid \frac{n_2}{g} (d_1+1)$. Since $\gcd(\frac{n-n_2}{g}, \frac{n_2}{g}) = 1$, it must be that $\frac{n-n_2}{g} \mid (d_1+1)$. Thus, there exists an integer $k \ge 1$ such that:
        \begin{equation}\label{eq:d1_k}
            d_1 = k \frac{n-n_2}{g} - 1.
        \end{equation}
        Substituting~\eqref{eq:d1_k} into (iii') yields $k \frac{n-n_2}{g} - 1 + n_2 \le n-1$, implying that $k \le g$.  Thus, $k \in \{1, 2, \dots, g\}$.
        
        We now apply the strict inverse dominance condition. Substituting~\eqref{eq:d1_k} into $(n-2n_2)d_1 > n_2$ yields:
        \begin{equation*}
            (n-2n_2)\left(k \frac{n-n_2}{g} - 1\right) > n_2.
        \end{equation*}
        Expanding and rearranging the terms gives $k \frac{(n-2n_2)(n-n_2)}{g} > n - n_2$. Since $n_2 < n$, we can divide both sides by the positive term $(n-n_2)$ to find the restricted range for $k$:
        \begin{equation}\label{eq:k_bound}
            k \frac{n-2n_2}{g} > 1.
        \end{equation}
        We evaluate the number of valid integers $k$ based on the ratio $\frac{n-2n_2}{g}$:
        \begin{itemize}
            \item If $2n_2 \ge n$, the term $(n-2n_2)$ is non-positive. Because $k \ge 1$, the left side of~\eqref{eq:k_bound} is $\le 0$, so the inequality is never satisfied, yielding $0$ candidates.
            \item If $\frac{n-2n_2}{g} = 1$, inequality~\eqref{eq:k_bound} becomes $k \cdot 1 > 1$, implying that $k \ge 2$. Thus $k \in \{2, \dots, g\}$, yielding $g-1$ candidates.
            \item If $\frac{n-2n_2}{g} > 1$, inequality~\eqref{eq:k_bound} requires $k \ge 1$, which is unconditionally true for all valid $k \in \{1, \dots, g\}$, yielding $g$ candidates.
        \end{itemize}
    \end{claimproof}

    \begin{claim}\label{claim:En2_table}
        Applying the boundary condition (iv) and the parity constraint (i), or their mirrored counterparts, reduces $C_1(n,n_2)$ and $C_2(n,n_2)$ into the exact valid counts $E_1(n, n_2)$ and $E_2(n, n_2)$ detailed in Table~\ref{table:En_components}. 
    
        \begin{table}[ht]
            \centering
            \renewcommand{\arraystretch}{1.5}
            \begin{tabular}{|c|c|c|c|c|c|}
                \hline
                $n$ & $n_2$ & Condition & $E_1(n,n_2)$ & $E_2(n,n_2)$ & $E(n,n_2)$ \\
                \hline
                \multirow{3}{*}{Even} & \multirow{3}{*}{Even} & $2n_2 \ge n$ & $g+1$ & $0$ & $g+1$ \\
                & & $(n-2n_2)/g = 1$ & $2$ & $g-1$ & $g+1$ \\
                & & $(n-2n_2)/g > 1$ & $1$ & $g$ & $g+1$ \\
                \hline
                Even & Odd & --- & $0$ & $0$ & $0$ \\
                \hline
                \multirow{4}{*}{Odd} & \multirow{4}{*}{Even} & $n_2 = n-1$ & $g$ & $0$ & $g$ \\
                & & $n/2 < n_2 \le n-3$ & $g+1$ & $0$ & $g+1$ \\
                & & $(n-2n_2)/g = 1$ & $2$ & $(g-1)/2$ & $(g+3)/2$ \\
                & & $(n-2n_2)/g > 1$ & $1$ & $(g+1)/2$ & $(g+3)/2$ \\
                \hline
                \multirow{3}{*}{Odd} & \multirow{3}{*}{Odd} & $2n_2 > n$ & $(g+1)/2$ & $0$ & $(g+1)/2$ \\
                & & $(n-2n_2)/g = 1$ & $1$ & $g-1$ & $g$ \\
                & & $(n-2n_2)/g > 1$ & $0$ & $g$ & $g$ \\
                \hline
            \end{tabular}
            \caption{Valid edge count contributions for a fixed pair $(n, n_2)$.}
            \label{table:En_components}
        \end{table}
    \end{claim}
    \begin{claimproof}
        We filter the pre-parity candidate pools $C_1(n, n_2)$ and $C_2(n, n_2)$ using the two remaining constraints of Theorem~\ref{theorem:when_biregular_simple_graph}: the boundary condition (iv) or (iv') and the parity constraint (i), which is the same as (i').

        \noindent\textbf{Step 1 - The Boundary Condition:}
        \begin{description}
            \item[$C_2$ - ] 
                Condition (iv') is $n_1(n_1-1) \ge n_1 d_1 - n_2 d_2$. As condition (iii') $n_1 \ge d_2$ already holds for the $C_2$ candidates, it suffices to check condition (iv') with the $n_1-1$ factor on the left replaced by $d_2-1 = d_1$. As this condition always holds, condition (iv') does not reject any of the $C_2$ candidates.

            \item[$C_1$ - ] 
                Condition (iv) is $n_2(n_2-1) \ge n_2 d_2 - n_1 d_1$, is equivalent to $n_2(n_2-2) \ge d_1 (2n_2 - n)$.
                \begin{itemize}
                    \item 
                        If $2 \le n_2 < n/2$ then the right hand side is $\le 0$ while the left hand side is $\ge 0$, so the condition always holds.
                    \item 
                        If $n/2 \le n_2 \le n-2$ then the factor $2n_2-n$ is non-negative, and increasing $d_1$ would yield a stricter condition. As condition (iii) $n_2 \ge d_1$ holds for all $C_1$ candidates. Replacing $d_1$ by $n_2$ and dividing by $n_2$ yields the condition $n_2-2 \ge 2n_2 - n$, that is satisfied for all $C_1$ candidates.
                    \item 
                        If $n_2=1$ then the candidates are $d_1 \in \{0,1\}$. 
                        Condition (iv) evaluates to $-1 \ge d_1 (2-n)$ which holds for $d_1=1$ as $n \ge 3$, but is violated for $d_1=0$. 
                        However, as $d_1=0$ is rejected by condition (i) requiring $n_1d_1 + n_2d_2$ to be even, we may ignore its rejection by (iv), knowing it will not pass the parity sieve.
                    \item 
                        If $n_2=n-1$ then condition (iv) is equivalent to $(n-1)(n-3) \ge d_1(2n_2-n) = d_1(n-2)$. As $g=\gcd(n,n-1)=1$, the $C_1$ candidates are $d_1 = m\frac{n_2}{g} = m(n-1)$ for $m \in \{0,1\}$. Then of the two candidates, only $m=0$ is accepted.
                \end{itemize}
                Overall, condition (iv) holds for all $C_1$ candidates with the exception of $n_2=n-1$, where one of the two candidates are rejected.
        \end{description}

        \noindent\textbf{Step 2: The Parity Sieve (i).}
        Condition (i) requires the sum of degrees, $n_1d_1 + n_2d_2 = nd_1 + n_2$, to be even. Note that as this condition is symmetric, (i) and (i') are the same. 
        We evaluate this condition systematically across the four parity combinations of $n$ and $n_2$:

        \begin{itemize}
            \item \textbf{$n$ is Even, $n_2$ is Odd:} The term $nd_1 + n_2$ is always odd, so condition (i) fails for all candidates, yielding $E_1 = E_2 = 0$.
            
            \item \textbf{$n$ is Even, $n_2$ is Even:} The term $nd_1 + n_2$ is inherently even. Condition (i) is unconditionally satisfied. All candidates survive, meaning $E_1 = C_1$ and $E_2 = C_2$.
            
            \item \textbf{$n$ is Odd, $n_2$ is Even:} We require $nd_1 + n_2 \equiv d_1 \equiv 0 \pmod 2$, meaning $d_1$ must be even. 
            For $C_1$, we have $d_1 = m \frac{n_2}{g}$. Since $n$ is odd, its divisor $g$ is odd, making $\frac{n_2}{g}$ even. Thus, $d_1$ is always even. All $C_1$ candidates survive, yielding $E_1 = C_1$, except for $n_2=n-1$ where $E_1 = C_1 - 1$.
            For $C_2$, we have $d_1 = k \frac{n-n_2}{g} - 1$. Since $n-n_2$ is odd, the fraction $\frac{n-n_2}{g}$ is odd. For $d_1$ to be even, $k$ must be odd, so we need to filter the contiguous range of $k$ candidates for odd integers. For $(n-2n_2)/g = 1$, the range starts at $k=2$, so $E_2(n,n_2) = \lceil (g-1)/2 \rceil$. Otherwise, if $(n-2n_2)/g > 1$, the range starts at $k=1$, so $E_2(n,n_2) = \lceil g/2 \rceil$.
            
            \item \textbf{$n$ is Odd, $n_2$ is Odd:} We require $nd_1 + n_2 \equiv d_1 + 1 \equiv 0 \pmod 2$, meaning $d_1$ must be odd. 
            For $C_2$, $d_1 = k \frac{n-n_2}{g} - 1$. Because $n-n_2$ is even and $g$ is odd, the fraction $\frac{n-n_2}{g}$ is even. Thus, $d_1$ is always odd, so all $C_2$ candidates survive, yielding $E_2 = C_2$.
            For $C_1$, $d_1 = m \frac{n_2}{g}$. Because $n_2$ is odd, $\frac{n_2}{g}$ is odd. For $d_1$ to be odd, $m$ must be odd. Filtering the contiguous sequence $m \in \{0, \dots, C_1-1\}$ for odd integers yields $\lfloor C_1/2 \rfloor$.
        \end{itemize}

        Summing the filtered contributions $E_1(n, n_2) + E_2(n, n_2)$ perfectly resolves the piecewise boundaries. For instance, in the Odd/Even case, for any odd $g$, the algebraic sum $2 + \lceil (g-1)/2 \rceil$ gracefully collapses to $(g+3)/2$. This confirms the values compiled in Table~\ref{table:En_components}.
    \end{claimproof}    

    With the combinatorial behavior fully mapped per $n_2$, we evaluate the total number of distinct simple graph minimizers $\sum_{n_2=1}^{n-1} E(n, n_2)$ across the density spectrum.

    \begin{claim}\label{claim:sum_even}
    For any even integer $n \ge 3$, the total sum evaluates to $E(n) = 2P(n/2) - \frac{n}{2} - 1$.
    \end{claim}
    \begin{claimproof}
        By Table~\ref{table:En_components}, if $n$ is even and $n_2$ is odd, the number of valid edge configurations is identically zero. Thus, the total sum for $E(n)$ strictly counts the contributions where $n_2$ is even. 
        
        For even $n$ and even $n_2$, the table establishes that $E(n, n_2) = g + 1 = \gcd(n, n_2) + 1$, completely independent of the dominance threshold. Because $n$ and $n_2$ are both even, we can define $n = 2N$ and $n_2 = 2j$. As $n_2$ ranges over the even integers in the open interval $(0, n)$, the index $j$ ranges from $1$ to $N-1$. 
        
        We can extract the common factor of 2 from the greatest common divisor, yielding \\ 
        $\gcd(n, n_2) = \gcd(2N, 2j) = 2\gcd(N, j)$. Substituting this into the total sum gives:
        \begin{equation}\label{eq:sum_even_N}
            E(n) = \sum_{j=1}^{N-1} \left( 2\gcd(N, j) + 1 \right) = 2 \left( \sum_{j=1}^{N-1} \gcd(N, j) \right) + N-1.
        \end{equation}
        Recall that Pillai's arithmetical function is defined over the full domain up to the integer, $P(N) = \sum_{j=1}^N \gcd(N, j)$. Our summation strictly stops at $N-1$, meaning it excludes the final term where $j=N$. Because $\gcd(N, N) = N$, the sum up to $N-1$ evaluates to exactly $P(N) - N$.
        
        Substituting this back into our split summation alongside the constant sum evaluates to:
        \begin{equation*}
            E(n) = 2 \left( P(N) - N \right) + N - 1 = 2P(N) - N - 1 = 2P(n/2) - \frac{n}{2} - 1.
        \end{equation*}
    \end{claimproof}    
    
    \begin{claim}\label{claim:sum_odd}
        For any odd integer $n \ge 3$, the total sum evaluates to $E(n) = \frac{3}{4}P(n) + \frac{n - f(n) - 8 + \chi_4(n)}{4}$.
    \end{claim}
    \begin{claimproof}
        For an odd integer $n$, the sum of valid edge configurations spans the full spectrum $1 \le n_2 \le n-1$ as defined in Table~\ref{table:En_components}.  
        Since $g$ is odd, we have $\lfloor (g+1)/2 \rfloor = (g+1)/2$. Also, in order to avoid an extra case, we regard $E(n,n-1)$ as $g+1$ and then subtract $1$ from the total.
        Therefore,
        \begin{equation*}
            E(n) = \sum_{\substack{n/2 < n_2 \\ n_2 \text{ even}}} (g+1) + \sum_{\substack{n/2 < n_2 \\ n_2 \text{ odd}}} \frac{g+1}{2}
                 + \sum_{\substack{n/2 > n_2 \\ n_2 \text{ even}}} \frac{g+3}{2} + \sum_{\substack{n/2 > n_2 \\ n_2 \text{ odd}}} \!\!g \,\,-\,\, 1.
        \end{equation*}

        Let $B(n)$ denote the baseline sum, where we naively extrapolate the $n/2 < n_2$ behavior over the entire range $1 \le n_2 \le n-1$:
        \begin{equation}\label{eq:baseline_odd}
            B(n) = \sum_{n_2 \text{ even}} (g+1) + \sum_{n_2 \text{ odd}} \frac{g+1}{2} = \sum_{n_2 \text{ even}} g + \frac{n-1}{2} + \frac{1}{2}\sum_{n_2 \text{ odd}} g + \frac{n-1}{4}.
        \end{equation}
        Because $n$ is odd, the map $n_2 \mapsto n - n_2$ is a bijection on the set $\{1, \dots, n-1\}$ that exchanges the even numbers with the odd numbers while preserving the greatest common divisor, $\gcd(n, n-n_2) = \gcd(n, n_2)$. This implies that $\sum_{n_2 \text{ even}} g = \sum_{n_2 \text{ odd}} g = \frac{1}{2} \sum_{n_2=1}^{n-1} \gcd(n, n_2) = \frac{1}{2}(P(n) - n)$. Substituting this into~\eqref{eq:baseline_odd} yields:
        \begin{equation*}
            B(n) = \frac{3}{2} \left( \frac{P(n) - n}{2} \right) + \frac{3(n-1)}{4} = \frac{3}{4}(P(n) - 1).
        \end{equation*}

        Next, we consider the difference $E(n) + 1 - B(n)$, where the sums for $n/2 < n_2$ cancel out. In the lower half of the $n_2$ range, we have $\frac{g+3}{2} - (g+1) = -\frac{g-1}{2}$ for even $n_2$ and $g - \frac{g+1}{2} = \frac{g-1}{2}$ for odd $n_2$. Therefore, we may write:
        \begin{align}
            E(n)      &= B(n) + \Delta(n) - 1 = \frac{3}{4}(P(n) - 1) + \Delta(n) - 1 \label{eq:E_odd}\\
            \Delta(n) &= \sum_{n_2 < n/2} (-1)^{n_2 - 1} \frac{g-1}{2}.                 \nonumber
        \end{align}
                
        To evaluate $\Delta(n)$ algebraically, we double it and split the summation:
        \begin{equation}\label{eq:delta_split}
            2\Delta(n) = \sum_{k=1}^{(n-1)/2} \gcd(n,k)(-1)^{k-1} - \sum_{k=1}^{(n-1)/2} (-1)^{k-1}.
        \end{equation}
        The second sum evaluates strictly to $\frac{1 - \chi_4(n)}{2}$. For the first sum, let $S_1 = \sum_{k=1}^{(n-1)/2} \gcd(n,k)(-1)^{k-1}$. We apply the identity $\gcd(n,k) = \sum_{d \mid n, d \mid k} \phi(d)$ and interchange the order of summation:
        \begin{equation}\label{eq:S1_swap}
            S_1 = \sum_{d \mid n} \phi(d) \sum_{\substack{1 \le k \le (n-1)/2 \\ d \mid k}} (-1)^{k-1}.
        \end{equation}
        Let $k = md$. Because $n$ is odd, any divisor $d \mid n$ must also be odd. Consequently, the product $md$ shares the same parity as $m$, allowing us to simplify $(-1)^{md-1} = (-1)^{m-1}$. The inner sum becomes: $\sum_{m=1}^{\lfloor \frac{n-1}{2d} \rfloor} (-1)^{m-1}$, evaluating to $1$ if the upper limit is odd, and $0$ otherwise. Because $d$ divides $n$, the fraction $n/d$ is an odd integer, and the upper limit simplifies exactly without the floor function to $\frac{n/d - 1}{2}$. This value is odd if and only if $n/d \equiv 3 \pmod 4$. 
        
        Therefore, the inner sum acts as a filter, allowing us to write
        \begin{equation}\label{eq:S1_filtered}
            S_1 = \sum_{\substack{d \mid n \\ n/d \equiv 3 \pmod 4}} \phi(d).
        \end{equation}
        We can express this filtered sum algebraically by comparing Gauss's theorem for the totient function, $n = \sum_{d \mid n} \phi(d)$, with our Dirichlet convolution, $f(n) = \sum_{d \mid n} \chi_4(n/d)\phi(d)$. 
        By substituting~\eqref{eq:S1_filtered}, we find that $S_1 = \frac{n - f(n)}{2}$. 
        
        Combining $S_1$ and the second sum back into~\eqref{eq:delta_split} yields:
        \begin{equation*}
            2\Delta(n) = S_1 - \frac{1 - \chi_4(n)}{2} = \frac{n - f(n) - 1 + \chi_4(n)}{2}.
        \end{equation*}

        Substituting into \eqref{eq:E_odd} yields the final result
        \begin{equation*}
            E(n) = \frac{3}{4}(P(n) - 1) + \Delta(n) - 1 = \frac{3}{4}P(n) + \frac{n - f(n) - 8 + \chi_4(n)}{4}.
        \end{equation*}
        This completes the evaluation of the total sum for odd $n$.
    \end{claimproof}    
    Claims~\ref{claim:sum_even} and~\ref{claim:sum_odd} together establish the exact piecewise formula for $E(n)$.
\end{proof}

We now proceed to prove the primary properties of this function as stated in Theorem~\ref{theorem:num_edges_simple_graph}.

\begin{proof}[Proof of Theorem~\ref{theorem:num_edges_simple_graph}]
    We evaluate the three clauses of the theorem sequentially.
    
    \begin{claim}\label{claim:En_lower_bound}
        For all $n \ge 3$, $E(n) \ge \lfloor\frac{3n-5}{2}\rfloor$.
    \end{claim}
    \begin{claimproof}
        To establish the absolute minimum of $E(n)$, we must lower bound Pillai's function $P(n)$ and upper bound the convolution $f(n)$. 
        Then we have:
        \begin{equation}\label{eq:pillai_bound}
            P(m) = \sum_{i=1}^m \gcd(i,m) \ge 1 \cdot (m-1) + m \cdot 1 = 2m - 1,
        \end{equation}
        since $\gcd(i,m) \ge 1$ for all $i \ge 1$ and since $\gcd(m,m)=m$. 
        Equality holds if and only if $\gcd(i,m)=1$ for all $i=1,\ldots,m-1$, meaning $m$ is a prime number.

        The convolution $f(n)$ is bounded by:
        \begin{align}
            f(n) &=   \sum_{d \mid n} \chi_4(d) \cdot \phi\left(\frac{n}{d}\right) 
                  \le \sum_{\substack{d \mid n\\d \neq n}} 1 \cdot \phi\left(\frac{n}{d}\right) + \chi_4(n) \cdot \phi\left(\frac{n}{n}\right) \nonumber \\
                 & = (n - \phi(1)) + \chi_4(n)         \label{eq:f_bound_general}
                 = n - 1 + \chi_4(n).
        \end{align}

        If $n$ is even then, as \eqref{eq:pillai_bound} yields $P(n/2) \ge n - 1$, we can write 
        \begin{equation}\label{eq:En_lower_bound_n_is_even}
            E(n) = 2P(n/2) - \frac{n}{2} - 1 \ge 2(n - 1) - \frac{n}{2} - 1 = \frac{3n - 6}{2} = \lfloor \frac{3n-5}{2} \rfloor.
        \end{equation}

        If $n$ is odd then applying $P(n) \ge 2n - 1$ and $f(n) \le n - 1 + \chi_4(n)$ yields:
        \begin{align}
            E(n) &= \frac{3}{4}P(n) + \frac{n - f(n) - 8 + \chi_4(n)}{4} \ge \frac{3(2n - 1)}{4} + \frac{n - (n - 1 + \chi_4(n)) - 8 + \chi_4(n)}{4} \nonumber \\
                 &= \frac{3n - 5}{2} = \lfloor \frac{3n-5}{2} \rfloor.    \label{eq:En_lower_bound_n_is_odd}
        \end{align}
    \end{claimproof}
    
    \begin{claim}\label{claim:En_tightness}
        The lower bound is tight if and only if $n=p$ or $n=2p$ for some prime number $p$.
    \end{claim}
    \begin{claimproof}
        By the derivations in Claim~\ref{claim:En_lower_bound}, the lower bound $E(n) \ge \lfloor \frac{3n-5}{2} \rfloor$ is tight if and only if the arithmetical bounds applied in the respective parity cases hold with strict equality. 

        For even $n$, the only requirement for the tight lower bound to apply is that $P(m) = 2m - 1$ for $m = n/2$. As this bound holds if and only if $n/2$ is prime, this establishes the $n=2p$ case.

        For odd $n$, the lower bound is tight exactly when both $P(n) = 2n-1$ and inequality \eqref{eq:En_lower_bound_n_is_odd} holds as equality, $f(n) = n - 1 + \chi_4(n)$.
        The requirement on Pillai's function implies that $n$ must be some prime number $p$.
        As for $n=p$, we have
        \begin{equation*}
            f(p) = \chi_4(1) \cdot \phi(p) + \chi_4(p) \cdot \phi(1) = p - 1 + \chi_4(p),
        \end{equation*}
        it follows that the bound for odd $n$ is tight if and only if $n$ is prime.
    \end{claimproof}

    \begin{claim}\label{claim:En_asymptotic}
        For an infinite family of highly composite, square-free integers, $E(n) = \Omega(n^{1+\frac{c}{\log \log n}} )$ for any constant $c$ such that $0 < c < \log 2$.
    \end{claim}

    \begin{claimproof}
        From the exact formula in Theorem~\ref{theorem:exact_En}, we observe that for all $n \ge 3$, the function $E(n)$ is bounded below by a linear fraction of Pillai's function minus a linear term in $n$. Specifically, $E(n) \ge \frac{3}{4}P(n) - O(n)$ for odd $n$, and $E(n) = 2P(n/2) - O(n)$ for even $n$. Therefore, it suffices to prove the bound for $P(n)$.

        The maximal order of Pillai's function is a well-established result in analytic number theory. 
        As detailed in T\'oth's comprehensive survey \cite{toth2010survey} (equation (24) and the remark at the end of Section 5.1), 
        $P(n)$ achieves its maximal growth along specific sequences of highly composite, square-free integers, satisfying the limit:
        \begin{equation}
            \limsup_{n \to \infty} \frac{\log(P(n)/n) \log \log n}{\log n} = \log 2.
        \end{equation}
        This implies that for such maximizing sequences of $n$, $P(n)$ scales asymptotically as:
        \begin{equation*}
            P(n) = n \cdot 2^{(1+o(1)) \frac{\log n}{\log \log n}} = n^{1 + (1+o(1)) \frac{\log 2}{\log \log n}}.
        \end{equation*}
        
        The claim follows as the expression on the right is strictly larger than $n^{1+\frac{c}{\log \log n}}$ for any $0 < c < \log 2$ and sufficiently large $n$.
    \end{claimproof}
    
    Noting that the expression in Claim~\ref{claim:En_asymptotic} is $\Omega(n \log n)$, this completes the proof of Theorem~\ref{theorem:num_edges_simple_graph}.
\end{proof}

\begin{remark}\label{remark:average_num_edges_simple_graph}
    While Theorem~\ref{theorem:num_edges_simple_graph} establishes the maximal order of $E(n)$ via highly composite numbers, we can also determine its average behavior. As detailed in T\'oth's survey~\cite[equation (16)]{toth2010survey}, the summatory function for Pillai's arithmetical function satisfies $\sum_{m \le x} P(m) = \frac{3}{\pi^2} x^2 \log x + O(x^2)$. Because $E(n)$ is bounded below by a linear fraction of $P(n)$ or $P(n/2)$ (minus a linear term), it immediately follows that the average order of $E(n)$ is $\Theta(n \log n)$. 
\end{remark}

\section{Conclusion and Open Questions}

This work investigates the question of how small can the spectral radius of a graph be, given its average degree, and possibly more information about its vertex degrees. Our approach is to relax the question by allowing graphs to be weighted. Namely, to have an adjacency matrix with entries that are not necessarily zero or one, while still maintaining the requirement that the matrix would be non-negative, symmetric, and with integral row sums. 
We use the convexity of the spectral radius $\rho$ as a function of the adjacency matrix, and perturbation theory tools that enable us to compute the derivative of $\rho$ under changes in the adjacency matrix. For the relaxed problem to weighted graphs, this provides an affirmative answer to a conjecture by Hong (1993) stating that for graphs minimizing the spectral radius the minimal and maximal degrees differ by at most one. Moreover, our results provide an improved lower bound to the spectral radius of a graph, given its average degree. As it turns out, for many cases this lower bound can be met by a simple graph, automatically implying that Hong's conjecture holds for these configurations.  In fact, our exact enumeration reveals that the number of such configurations grows super-linearly with the number of vertices on average.

We end with a list of questions for future research:
\begin{enumerate}
\item 
    Our original motivation for this work was to improve the Moore Bound for such graphs by finding a better lower bound on $\rho(B)$ given the average degree, as observed in \cite{hoory2024girth}. Moving forward, can similar results be obtained for the non-backtracking adjacency matrix $B$, as defined in \cite{angel2015non} and \cite{eisner2024entropy}?
\item 
    Can the result of Theorem~\ref{theorem:Mn1d1n2d2} be generalized to an arbitrary degree sequence? Can the minimizing weighted graph for a general degree sequence be characterized?
\item 
    Can Hong's conjecture be proved for additional $(n,e)$ pairs for which there is no simple graph meeting the weighted matrix lower bound of Theorem~\ref{theorem:Mne}?
\item 
    Can one get a better lower bound on $\rho(\G_{n,e})$, when the average degree $\overline{d}=2e/n$ is not in the list \eqref{eq:average_degree_sequence}? 
\end{enumerate}

\section{Acknowledgments}
    We would like to acknowledge that Gemini 3.1 Pro was instrumental in obtaining the closed formula in Theorem~\ref{theorem:exact_En}. 

\bibliographystyle{abbrv}
\bibliography{ref}

\end{document}